\documentclass[preprint,12pt]{elsarticle}

\journal{Computers and Mathematics with Applications}

\usepackage{subfigure}
\usepackage{graphicx}
\usepackage{amsmath}
\usepackage{amsthm}
\usepackage{amsfonts}
\usepackage{enumerate}
\usepackage[boxed,lined]{algorithm2e}
\usepackage{url}

\begin{document}

\begin{frontmatter}



\title{
Subspace-preserving sparsification of matrices with minimal
perturbation to the near null-space. Part I: Basics
}


\author{Chetan Jhurani}
\ead{chetan.jhurani@gmail.com}

\address{
	Tech-X Corporation\\
	5621 Arapahoe Ave\\
	Boulder, Colorado 80303, U.S.A.
}


\begin{abstract}

This is the first of two papers to describe a matrix sparsification
algorithm that takes a general real or complex matrix as input and produces
a sparse output matrix of the same size.  The non-zero entries in the output
are chosen to minimize changes to the singular values and singular vectors
corresponding to the near null-space of the input.  The output matrix is
constrained to preserve left and right null-spaces exactly.  The sparsity pattern of the
output matrix is automatically determined or can be given as input.

If the input matrix belongs  to a common matrix subspace, we prove that the
computed sparse matrix belongs to the same subspace.  This works without imposing
explicit constraints pertaining to the subspace. This property holds for the
subspaces of Hermitian, complex-symmetric, Hamiltonian, circulant,
centrosymmetric, and persymmetric matrices, and for each of the skew
counterparts.

Applications of our method include computation of reusable sparse
preconditioning matrices for reliable and efficient solution
of high-order finite element systems.  The second paper in
this series~\cite{cite:ssa_2} describes our open-source implementation,
and presents further technical details.

\end{abstract}

\begin{keyword}

Sparsification \sep
Spectral equivalence \sep
Matrix structure \sep
Convex optimization \sep
Moore-Penrose pseudoinverse

\end{keyword}

\end{frontmatter}



\newcommand{\matspace}[3]{\ensuremath{\mathbb{#1}^{#2 \times #3}} }

\newcommand{\Rnn}  {\matspace{\mathbb{R}}{n}{n}}
\newcommand{\Rmn}  {\matspace{\mathbb{R}}{m}{n}}
\newcommand{\Rmm}  {\matspace{\mathbb{R}}{m}{m}}

\newcommand{\Cnn}  {\matspace{\mathbb{C}}{n}{n}}
\newcommand{\Cmn}  {\matspace{\mathbb{C}}{m}{n}}
\newcommand{\Cmm}  {\matspace{\mathbb{C}}{m}{m}}
\newcommand{\Crn}  {\matspace{\mathbb{C}}{r}{n}}
\newcommand{\Crr}  {\matspace{\mathbb{C}}{r}{r}}


\newcommand{\vecspace}[2]{\ensuremath{\mathbb{#1}^{#2}} }

\newcommand{\Rn}   {\vecspace{\mathbb{R}}{n}}
\newcommand{\Cn}   {\vecspace{\mathbb{C}}{n}}

\newcommand{\Rm}   {\vecspace{\mathbb{R}}{m}}
\newcommand{\Cm}   {\vecspace{\mathbb{C}}{m}}

\newcommand{\reals} {\ensuremath{\mathbb{R}}}
\newcommand{\complex} {\ensuremath{\mathbb{C}}}


\newcommand{\norm}[1]{\ensuremath{\left| \left| #1 \right| \right|} }

\newcommand{\abs}[1]{\ensuremath{\left| #1 \right|} }

\newcommand{\pinv}[1]{\ensuremath{{#1}^{\dagger}} }
\newcommand{\inv}[1]{\ensuremath{{#1}^{-1}} }
\newcommand{\h}[1]{\ensuremath{{#1}^{op}} }

\newcommand{\nullsp}[1]{\ensuremath{\mathcal{N}(#1)} }
\newcommand{\rangsp}[1]{\ensuremath{\mathcal{R}(#1)} }

\newcommand{\patsym}[0]{\ensuremath{\mathcal{Z}} }
\newcommand{\pat}[1]{\ensuremath{\mathcal{Z}(#1)} }

\newcommand{\lagr}[0]{\ensuremath{\mathcal{L}} }

\newcommand{\partderiv}[2]{\ensuremath{\frac{\partial #1}{\partial #2}} }

\newcommand{\half}[0]{\ensuremath{\frac{1}{2}} }


\newcommand{\ordset}[3] {\ensuremath{\left\{ {#1}_{#2} \right\}_{#2 = 1}^{#3}}}


\newcommand{\Eq}[1]  {Equation~(\ref{#1})}
\newcommand{\Eqs}[1] {Equations~(\ref{#1})}
\newcommand{\Sec}[1] {Section~\ref{#1}}
\newcommand{\Fig}[1] {Figure~\ref{#1}}
\newcommand{\Alg}[1] {Algorithm~\ref{#1}}
\newcommand{\Rem}[1] {Remark~\ref{#1}}
\newcommand{\Prop}[1] {Property~\ref{#1}}
\newcommand{\Thm}[1] {Theorem~\ref{#1}}
\newcommand{\Def}[1] {Definition~\ref{#1}}
\newcommand{\Tab}[1] {Table~\ref{#1}}


\newtheorem{theorem}{Theorem}[section]
\newtheorem{lemma}[theorem]{Lemma}
\newtheorem{proposition}[theorem]{Proposition}
\newtheorem{corollary}[theorem]{Corollary}
\newtheorem{conjecture}[theorem]{Conjecture}
   
\theoremstyle{definition}
\newtheorem{definition}[theorem]{Definition}
\newtheorem{example}[theorem]{Example}
\newtheorem{examples}[theorem]{Examples}

\theoremstyle{remark}
\newtheorem{remark}[theorem]{Remark}



\section{Introduction}

We present and analyze a matrix-valued optimization problem formulated to
sparsify matrices while preserving the matrix null-spaces and certain
special structural properties.  Algorithms for sparsification of matrices
have been used in various fields to reduce computational and/or data
storage costs.  Examples and applications are given in 
\cite{
Spielman:2008:GSE:1374376.1374456,
Halko:2011:FSR:2078879.2078881,
Achlioptas:2007:FCL:1219092.1219097,
Arora:2006:FRS:2165230.2165259} and the references therein.  In general, these
techniques lead to small perturbations only in the higher end of the spectrum and don't
respect the null-space or near null-space.  This is intentional and not necessarily
a limitation of these algorithms.

We have different objectives for sparsification.  First, we want to preserve
input matrix structural properties related to matrix subspaces, if it has any.  Being Hermitian or
complex-symmetric are just a couple of examples of such properties.
Second, we want sparsification to minimally perturb the singular values and
both left and right singular vectors in the lower end of the spectrum.
Third, we want to exactly preserve the left and right null-spaces if there
are any.

Our main contribution in this paper is designing an algorithm that fulfills
all the three objectives mentioned above.  We also provide theoretical
results and numerical experiments to support our claims.  The algorithm
presented here is inspired from a less general but a similar sparsification
algorithm introduced in our earlier publication~\cite{cite:cj_tma_bj_jcp}.
That algorithm, in turn, was inspired from the results presented in~\cite{AusBrez}.
Understanding the previous versions of the algorithm is not a prerequisite
for understanding the one presented here.

There has been a related work where one solves an optimization problem to
find a sparse preconditioner~\cite{cite:Greenbaum}.  However, the function to
be minimized there was non-quadratic and non-smooth, the method worked only
for symmetric positive-definite matrices, the sparsity patterns
were supposed to be chosen in advanced, and the method was quite slow.  Our algorithm
does not produce {\em the} optimal sparsified matrix as measured by condition number.
Nonetheless, it does not suffer from these limitations and produces preconditioning
matrices that are competitive in terms of condition number.

Our first objective, of maintaining structure, is important for aesthetic as
well as practical reasons.  If an algorithm produces an asymmetric matrix
by sparsifying a symmetric matrix, it just doesn't look like an
``impartial'' algorithm.  On the practical side, if a property like
symmetry were lost, one would have to resort to data structures and
algorithms for asymmetric matrices which, in general, will be slower than
those for symmetric matrices.  The second and third objectives -- not
disturbing the near null-spaces and the corresponding singular values while also
preserving the null-spaces -- are important when the inverse of the matrix
is to be applied to a vector rather than the matrix itself.  Then, the
inverse, or the pseudo-inverse, of the sparsified matrix can then be used in
place of the inverse of the dense matrix without incurring a large error
due to sparsification.  This also means that the sparsified matrix can be
used to compute a preconditioner with the likelihood that computation with it will
be less expensive due to the sparsity we introduce.

To fulfill these objectives, we design and analyze a matrix-valued,
linearly constrained convex quadratic optimization problem.  It works for a
general real or complex matrix, not necessarily square, and produces a sparse
output matrix of the same size.   A suitable pattern for sparsity is
also computed.  The user can control the sparsity level easily.  We prove that the output
sparse matrix automatically belongs to certain subspaces without imposing any
constraints for them if the input matrix belongs to them.  In particular,
this property holds for Hermitian, complex-symmetric, Hamiltonian,
circulant, centrosymmetric, and persymmetric matrices and also for each of
the skew counterparts.

We briefly mention that the rows and columns of the input matrix should be
well-scaled for the best performance of our sparsity pattern computation method.
A rescaling by pre- and post-multiplication by a diagonal matrix leads to
a well-scaled input matrix~\cite{cite:ruiz}.  We intend to highlight the effects
of scaling to our algorithm in a future publication.

In practice, we are concerned with matrices of size less than a few
thousands.  Applications of our method include computation of reusable
sparse preconditioning matrices for reliable and efficient (in time
consumed and memory used) techniques for solutions of high-order finite
element systems, where the element stiffness matrices are the candidates
for sparsification \cite{cite:cj_tma_bj_jcp}.  We want to stress that the
sparsification process is somewhat expensive because we need to solve a
matrix-valued problem and we want to respect the null-space and the near
null-space rather than the other end of the spectrum.  However, we can
reuse the computed sparse matrices for multiple elements.  We have also
modified the algorithm slightly to speed it up by orders of magnitude and
this is described in the next paper in the series~\cite{cite:ssa_2}.

Intuitively, any sparsification algorithm with objectives like ours will be
more useful when a large fraction of the input matrix entries has small
but non-zero magnitude relative to the larger entries.  Typical high-order finite element matrices have
this feature due to approximate cancellation that happens when high order
polynomials are integrated in the bilinear form.  On the other extreme, Hadamard matrices are poor candidates
for sparsification like ours.

Of course, special techniques in some situations can
generate sparse stiffness matrix without any explicit
sparsification~\cite{cite:Beuchler}.  However, our objective is to be as
general as possible and let the sparsification handle any kind of input
matrix.

The second paper~\cite{cite:ssa_2} describes version 1.0 of TxSSA, which is
a library implemented in C++ and C with interfaces in C++, C, and MATLAB. It
is an implementation
of our sparsification algorithm.  TxSSA is an acronym for Tech-X
Corporation Sparse Spectral Approximation.  The code is distributed under the
BSD 3-clause open-source license and is available here.
\begin{center}
\url{http://code.google.com/p/txssa/}
\end{center}

Here is an outline of the paper.  In~\Sec{sec:opt_problem}, we describe the
optimization problem, provide a rationale for the choices we make, and
prove that it is well-posed.  In~\Sec{sec:lp_pattern}, we describe the
algorithm for computing the sparsity pattern and state some of its
properties.  In~\Sec{sec:bounds}, we prove a few theoretical bounds that
relate the output matrix to the input matrix and other input parameters.
In~\Sec{sec:preserve}, we prove that our algorithm preserves many important
matrix subspaces.  Finally, in~\Sec{sec:num_res}, we show some basic
numerical results.  Detailed results and other practical concerns are
given in the second paper of this series~\cite{cite:ssa_2}.

\section{An optimization problem for sparsification}
\label{sec:opt_problem}

We work with rectangular complex matrices and complex vectors.  Simplifications
to the real or square case are straightforward.

\subsection{Notation}
\label{sec:notation}

Let \Cmn denote the vector space of rectangular complex matrices with $m$
rows and $n$ columns.  Let $\Cm$ denote the vector space of complex vectors
of size $m$. The superscripts `T', `*', and `$\dagger$' denote transpose,
conjugate transpose, and the Moore-Penrose pseudoinverse respectively.  A
bar on top of a quantity implies complex conjugate.  Let \nullsp{\cdot}
denote the null-space and \rangsp{\cdot} denote the range-space of a
matrix.  The default norm used for matrices is the Frobenius norm.  For
vectors the Euclidean norm is the default.  Otherwise norms have specific
subscripts. As usual, the identity matrix of size $n$ is denoted by $I_n$.
The symbol $I$ is used in case the size is easily deducible.

Let $A \in \Cmn$ be the input matrix to be sparsified.  The matrix $A$ can be
expressed using the singular value decomposition (SVD) as $U \Sigma V^*$.
Let $r := \text{rank}(A)$.  The singular values of $A$ are
\ordset{\sigma}{i}{\min(m,n)} sorted in non-increasing order, the right
singular vectors are \ordset{v}{i}{n}, and the left singular vectors are
\ordset{u}{i}{m}.  We divide $U$ and $V$ into two blocks each based on
the rank $r$. We get
$$
A = U \Sigma V^* = \left[ U_1 \; U_2 \right] \Sigma \left[ V_1 \; V_2 \right]^*
$$
where $U_1$ and $V_1$ each have $r$ columns.  Similarly,
$$
\Sigma =
\begin{bmatrix}
\Sigma_r & 0 \\
0 & 0 \\
\end{bmatrix}
$$
where $\Sigma_r$ is of size $r\times r$, diagonal, and invertible unless
$r = 0$.  The columns in $U_1$ and $V_1$ form orthonormal basis for
$\rangsp{A}$ and $\rangsp{A^*}$, respectively.  Similarly, the columns in
$U_2$ and $V_2$ form orthonormal basis for $\nullsp{A^*}$ and $\nullsp{A}$,
respectively.

Using the rank-nullity theorem, we get
$$
p_R := \text{dim}(\nullsp{A}) = n - r
$$
where $p_R$ is the ``right nullity''. Similarly,
$$
p_L := \text{dim}(\nullsp{A^*}) = m - r
$$
where $p_L$ is the ``left nullity''.  Thus,
$\nullsp{A^*} \in \matspace{\complex}{m}{p_L}$ and
$\nullsp{A} \in \matspace{\complex}{n}{p_R}$.

Let $\kappa = \kappa(A) := \norm{A}_2 \norm{\pinv{A}}_2$ be a condition
number for rank-deficient matrices.  Obviously, $\kappa =
\frac{\sigma_1(A)}{\sigma_r(A)} \geq 1$ for non-zero matrices and it
generalizes the usual condition number, which is typically used for square non-singular matrices.

Let $\patsym = \pat{A} \in \Rmn$ be a matrix that denotes a sparsity
pattern corresponding to $A$. It contains zeros and ones only. A zero in
\patsym means the entry at the corresponding position in the output matrix
is fixed to be zero.
A one in \patsym means the entry at corresponding position will be allowed
to vary and can be non-zero.  The pattern matrix is always real even for
complex input matrices.  We do not create separate patterns for real and
imaginary parts of a complex matrix.  The rationale behind this choice will
be explained later in \Rem{rmk:real_pattern} in~\Sec{sec:Zmat}, where we
will also provide a concrete algorithm for computing \pat{A}.

\subsection{A misfit functional}
\label{sec:misfit}
Let $X = X(A) \in \Cmn$ be a matrix produced by sparsification of $A$.
Fix $X_{ij} = 0$ if $\patsym_{ij} = 0$.
When the sparsity pattern is fixed, $X$ belongs to a subspace of $\Cmn.$
 Here, $i$ and $j$ are indices
such that $1 \le i \le m$ and $1 \le j \le n$.

We define an SVD based quadratic ``misfit'' functional $J \in \reals$ to
specify a difference between input matrix $A$ and an arbitrary matrix
$X$.
\begin{definition}
\begin{equation}
\label{eq:J_SVD}
J = J(X; A) :=
\half \sum_{i=1}^{r} \frac{1}{\sigma_i^2} \norm{(X-A) v_i}^2_2 +
\half \sum_{i=1}^{r} \frac{1}{\sigma_i^2} \norm{(X^* - A^*) u_i}^2_2
\end{equation}
\end{definition}
This misfit quantifies the action of the unknown matrix $X$ on the singular
vectors of $A$ and penalizes the differences in near null-space with larger
``weights''.  This is just a symbolic expression for defining $J$.  We shall see in
\Sec{sec:no_svd} that the SVD of $A$ does not have to be computed to compute $J$ and
its derivatives.

We now pose a linearly constrained quadratic optimization problem to
compute $X$.
\begin{align}
\label{eq:min_J}
\min_{X} J(X; A) & \text{ such that } \nonumber\\
&\nullsp{A} \subseteq \nullsp{X}, \nonumber\\
&\nullsp{A^*} \subseteq \nullsp{X^*}, \text{ and } \nonumber\\
&X \text{ has a specified sparsity pattern } \pat{A}
\end{align}
It is readily seen that the constraints are linear homogeneous equality
constraints and $J$ is a quadratic function in the entries of $X$.
Additionally, the misfit $J(X; A)$ is quadratic and bounded below by zero
for any $X$ whether the constraints are imposed or not.  This is one way to
see that it is also a convex function of $X$.  We can rewrite the
minimization problem as follows.
\begin{equation}
\label{eq:min_J_2}
\min_{X} J(X; A) \text{ such that } X V_2 = 0, X^* U_2 = 0, \text{ and }
X_{ij} = 0 \text{ if } (\pat{A})_{ij} = 0.
\end{equation}

\subsection{Rationale for the two-term misfit}

We had posed a similar minimization problem in~\cite{cite:cj_tma_bj_jcp}.
We used it to construct sparse preconditioners for high-order finite
element problems.  However, the input matrix
$A$ had to be real and symmetric.  The output matrix was specifically
constrained to be symmetric.  The misfit there had only one term, rather than two
as shown here, because symmetry of $A$ and $X$ implied that the two terms
were equal and hence one was redundant.  For the same reason only the right
null-space was imposed. The left null-space constraint, or equivalently,
the constraint related to the conjugate-transpose matrix's null-space,
became redundant.  Generalization to asymmetric matrices is one of the
contributions of the current research.

We now work with an arbitrary input matrix (complex, rectangular) and thus
make the setting completely general.  This is done by working with the
singular values and singular vectors rather than eigenvalues and eigenvectors.
We have also avoided explicit imposition of structural constraints, like matrix symmetry.

There are multiple reasons behind the choice of not using explicit
structural constraints.  In many cases, it is indeed possible to impose
linear equality constraints to impose that $X$ belong to a certain subspace
of matrices.  However, explicit specification would lead to specialized
code for any new subspace.  For example, special handling of Hermitian,
complex-symmetric, Hamiltonian, and many other matrices would be needed.
Although imposing such constraints individually is
simple, explicitly imposing such constraints doesn't tell us what to do in
case of matrices without common structures.  Consider a general matrix $A$
such that $\norm{A - A^*} \approx 0.001 \norm{A}$. It is reasonable to
expect and desire a similar small difference in $X$ and $X^*$.  However, in
general, we cannot impose an affine equality constraint to impose this.
Thus, even a tiny perturbation in input that destroys some structure
results in a completely different mathematical and computational problem
(assuming we were actually explicitly imposing $X = X^*$ when $A = A^*$).  These
considerations motivated us to design a new algorithm where no such
constraints are needed.  We now use a two-term misfit, two null-space
related constraint sets, and abandon such explicit constraints even in
cases where it is meaningful.   We show in~\Sec{sec:preserve} that such an algorithm still
preserves matrix structural properties whenever they are present.

\subsection{Avoiding the SVD}
\label{sec:no_svd}
We now show that the misfit in~\Eq{eq:J_SVD} can be expressed in terms of
the Moore-Penrose pseudoinverse of the input matrix $A$. This avoids using
its singular values and singular vectors for actual computation.  We
expressed the misfit using SVD first to show that the misfit penalizes
deviations in the lower end of the spectrum of $A$ and to motivate the
formulation.  See~\Sec{sec:notation} for the notation.

\begin{lemma}
Let $A, Y \in \Cmn$. Let $A = U \Sigma V^*$ be its SVD. Then
\begin{align*}
\sum_{i=1}^{r} \frac{1}{\sigma_i^2} \norm{Y v_i}^2 &= \norm{Y \pinv{A}}^2, \text{ and}\\
\sum_{i=1}^{r} \frac{1}{\sigma_i^2} \norm{Y^* u_i}^2 &= \norm{Y^* \pinv{(A^*)}}^2.
\end{align*}
\end{lemma}
\begin{proof}
Using the SVD based expression for the Moore-Penrose pseudoinverse
and the fact that the Frobenius matrix norm is invariant under multiplication by a
unitary matrix, it is easily seen that
$$
\norm{Y \pinv{A}}^2 = \norm{Y V \pinv{\Sigma} U^*}^2 = \norm{Y V \pinv{\Sigma}}^2.
$$
The last term can be expanded to a summation using $v_i$,
the columns of $V$, and $\frac{1}{\sigma_i}$, the non-zero diagonal values
of $\pinv{\Sigma}$.  This results in the left hand side of the first
equality.  Thus the first equality holds.  The second equality can be
derived in the same fashion.
\end{proof}
The previous result and the property that $\norm{Y} = \norm{Y^*}$ are used to
express the misfit $J$ from~\Eq{eq:J_SVD} as
\begin{equation}
\label{eq:J_pinv}
J(X; A) = \half \norm{(X - A) \pinv{A}}^2_F + \half \norm{\pinv{A} (X - A)}^2_F
\end{equation}
This avoids the use of singular vectors and singular values and one
avoid using the SVD. The pseudoinverse or a reasonable approximation to
it can be computed by any algorithm that is suitable depending on the
conditioning and any known matrix properties.  By way of example, if the
rank-revealing QR algorithm reliably determines the numerical rank of $A$,
it can be used to compute an accurate pseudoinverse.  On the other hand if
$A$ were symmetric semi-definite too, the pivoted Cholesky will be faster
for computing the pseudoinverse.

This is a good place to mention that we are indeed making the assumption
that the condition number of the input matrix is not ``too large''.  This
assumption is made so that one can avoid SVD and use a faster but less
robust algorithm to compute the pseudoinverse.  More importantly, if the
condition number of $A$ were very large, its invariant spaces will be
highly sensitive to any perturbation applied to $A$.  Sparsifying might be
worthless even if SVD is used.  What is a ``too large'' condition number
can can be quantified after some numerical experimentation and will also
strongly depend on the relative number of zeros introduced via
sparsification.

\subsection{First-order optimality conditions}

We derive the first-order optimality conditions for the minimization
problem using the method of Lagrange multipliers.  Since the quadratic form
in~\Eq{eq:min_J} is convex for all $A$ (see~\Sec{sec:misfit}) and the set of
feasible points is convex, any stationary point will be a local minima.
Hence it is sufficient to consider solutions of the first-order optimality
conditions.

\subsubsection{The Lagrange multipliers}
\label{sec:lagrangian}
We introduce the Lagrange function $\lagr = \lagr(X, \Lambda_R, \Lambda_L,
\Lambda_{\patsym}; A)$, where $X$ is the unknown sparse matrix, $\Lambda_R
\in  \matspace{\complex}{m}{p_R}$ and $\Lambda_L \in
\matspace{\complex}{n}{p_L}$ are matrices of Lagrange multipliers
corresponding to the right and left null-space constraints, respectively.
To impose the sparsity pattern, we use a matrix $\Lambda_{\patsym} \in
\Cmn$, a matrix of Lagrange multipliers and zeros.  If $(\pat{A})_{ij} = 0$,
$(\Lambda_{\patsym})_{ij} = \mu_{ij} \in \complex$ else $(\Lambda_{\patsym})_{ij} = 0$.
\begin{align}
\lagr(X, \Lambda_R, \Lambda_L, \Lambda_{\patsym}; A) :=&
\half \norm{(X - A) \pinv{A}}^2 + \half \norm{\pinv{A} (X - A)}^2  \nonumber \\
+&\; \text{trace}(\Lambda_R^* X V_2)  \nonumber \\
+&\; \text{trace}(\Lambda_L^* X^* U_2)  \nonumber \\
+&\; \sum_{  \substack{ij \text{ where }\\ (\pat{A})_{ij} = 0}   } \mu_{ij} X_{ij} \nonumber 
\end{align}

\subsubsection{Derivative of the misfit}
\label{sec:J_deriv}
Before differentiating $\lagr$, we write the derivative of $J$ with respect
to $X$.  We get
\begin{align*}
\partderiv{J}{X} &= (X-A) \pinv{A} (\pinv{A})^* + (\pinv{A})^* \pinv{A} (X-A)\\
                 &= \left(X \pinv{A} (\pinv{A})^* + (\pinv{A})^* \pinv{A} X\right) - \left(A \pinv{A} (\pinv{A})^* + (\pinv{A})^* \pinv{A} A\right) \\
                 &= \left(X \pinv{A} (\pinv{A})^* + (\pinv{A})^* \pinv{A} X\right) - 2 (\pinv{A})^*
\end{align*}
The simplification done above, in the terms not involving $X$, can be
proved easily using the SVD based expression of the Moore-Penrose
pseudoinverse.

Using the relation between the ``vectorization operation''
and the Kronecker product, we get
$$
\text{vec}\left(\partderiv{J}{X}\right) = \left(
\pinv{A} (\pinv{A})^* \otimes I_m +
I_n \otimes (\pinv{A})^* \pinv{A}
 \right) \text{vec}(X)
$$
where $I_k$ is the identity matrix of size $k$.  We define the $mn \times
mn$ Kronecker sum matrix $\mathcal{A} = \mathcal{A}(A)$ as
\begin{equation}
\label{eq:J_deriv}
\mathcal{A} = \mathcal{A}(A) :=
\pinv{A} (\pinv{A})^* \otimes I_m +
I_n \otimes (\pinv{A})^* \pinv{A}
\end{equation}
Using the spectral properties of Kronecker sums, it is obvious that
$\mathcal{A}(A)$ is always Hermitian positive semi-definite. It is
Hermitian positive definite if and only if $A$ is full-rank.  We skip the
proofs.

\subsubsection{Derivative of the Lagrangian}
\label{sec:first_order_opt}
We differentiate the Lagrangian given in~\Sec{sec:lagrangian} to derive the
first order optimality conditions.  Instead of presenting the intermediate
steps in the derivation we write the final expressions only.  Appendix A
in~\cite{cite:cjphd} shows the intermediate steps for matrix-valued
derivatives.
\begin{align*}
\partderiv{\lagr}{X}         = 0 \implies & X \pinv{A} (\pinv{A})^* + (\pinv{A})^* \pinv{A} X + \Lambda_R V_2^* + U_2 \Lambda_L^* + \Lambda_{\patsym} = 2 (\pinv{A})^*\\
\partderiv{\lagr}{\Lambda_R} = 0 \implies & X V_2 = 0\\
\partderiv{\lagr}{\Lambda_L} = 0 \implies & X^* U_2 = 0\\
\partderiv{\lagr}{\mu_{ij}}  = 0 \implies & X_{ij} = 0 \text{ for } ij \text{ such that } (\pat{A})_{ij} = 0
\end{align*}

\subsection{Existence and global uniqueness of the minimizer}
\label{sec:exist_unique}

We show that the minimization problem posed in~\Eq{eq:min_J} and with the
linear system shown in~\Sec{sec:first_order_opt} has a unique solution for
any input matrix $A$ and any imposed sparsity pattern.  This is true even
if the sparsity pattern has no relation to the input matrix $A$.

\begin{lemma}
\label{lemma:J_min_existence}
A minimizer always exists for the minimization problem posed in
\Eq{eq:min_J} for an arbitrary imposed sparsity pattern.
\end{lemma}
\begin{proof}
The equality constraints for null-spaces and sparsity are linear and
homogeneous.  Thus, the set of feasible solutions is non-empty.  For
example, the zero matrix is a feasible element.  Since the quadratic form
is convex for all $X$ (see~\Sec{sec:misfit}), a minimizer always
exists~\cite{cite:convex_opt}.
\end{proof}

We now have to prove that there is a globally unique minimizer.  We do this first for
the full-rank $A$ and then for a non-zero rank-deficient $A$.

\begin{lemma}
\label{lemma:J_min_unique1}
If $A$ is full-rank, the minimization problem posed in~\Eq{eq:min_J} has a
globally unique minimizer for an arbitrary imposed sparsity pattern.
\end{lemma}
\begin{proof}
As mentioned in~\Sec{sec:J_deriv}, the quadratic form $J$ is strictly
convex on $\Cmn$ if $A$ is full-rank.  In particular, it is strictly convex
on the subspace of those matrices that satisfy the equality constraints.
Since the feasible set is convex and non-empty, the minimizer is
globally unique.
\end{proof}

\begin{lemma}
\label{lemma:J_min_unique2}
If $A$ is non-zero and rank-deficient, the minimization problem posed in~\Eq{eq:min_J}
has a globally unique minimizer for an arbitrary imposed sparsity pattern.
\end{lemma}
\begin{proof}
For rank-deficient matrices the Hessian of the quadratic form is merely
positive semi-definite on $\Cmn$.  But we show that it is positive definite
on the subspace of matrices that satisfy the null-space related
constraints.  Showing this will ensure that the Hessian is positive
definite on the subspace of matrices satisfying all the constraints (which
includes sparsity constraints also).  This, in turn, means that there is a
globally unique minimizer even for the rank-deficient case.

Using the fact that $U$ and $V$ are unitary, it can be easily shown that
all matrices $X$ that satisfy the null-space constraints can be written as
a linear combination of $r^2$ basis matrices using arbitrary complex
coefficients
$\left\{\alpha_{ij}\right\}_{i,j=1}^{r}$.
$$
X = \sum_{i=1}^{r} \sum_{j=1}^{r} \alpha_{ij} \; u_i v_j^*
$$
As defined earlier, $u_i$ and $v_j$ are the left and right singular vectors
of $A$ respectively.

Consider the action of the Kronecker sum matrix $\mathcal{A}(A)$
(\Eq{eq:J_deriv}) on the vectorized basis matrices $u_i v_j^*$.
\begin{align*}
\mathcal{A}(A) \text{vec}(u_i v_j^*) &= \left( \pinv{A} (\pinv{A})^* \otimes I_m + I_n \otimes (\pinv{A})^* \pinv{A} \right) \text{vec}(u_i v_j^*)\\
&= \text{vec}\left( u_i v_j^* \pinv{A} (\pinv{A})^* + (\pinv{A})^* \pinv{A} u_i v_j^* \right)\\
&= \text{vec}\left( u_i \frac{u_j^*}{\sigma_j} (\pinv{A})^* + (\pinv{A})^* \frac{v_i}{\sigma_i} v_j^* \right)\\
&= \text{vec}\left( u_i \frac{1}{\sigma_j^2} v_j^* + u_i \frac{1}{\sigma_i^2} v_j^* \right)\\
&= \left( \frac{1}{\sigma_i^2} + \frac{1}{\sigma_j^2} \right) \text{vec}(u_i v_j^*)
\end{align*}
Since both $i,j \le r$, $\sigma_i > 0$ and $\sigma_j > 0$.  Thus,
$\text{vec}(u_i v_j^*)$ are eigenvectors of the Hessian matrix and the
corresponding eigenvalues are positive.  Hence, the Hessian restricted to
the subspace of matrices that satisfy the null-space constraints is
positive definite.
\end{proof}

\begin{theorem}
\label{thm:unique}
A globally unique minimizer exists for the minimization problem posed in
\Eq{eq:min_J} for an arbitrary imposed sparsity pattern.
\end{theorem}
\begin{proof}
This is a consequence of the three lemmas proved earlier $-$ Lemmas
\ref{lemma:J_min_existence}, \ref{lemma:J_min_unique1},
\ref{lemma:J_min_unique2}, and the easily seen fact that if $A = 0$, then
$X = 0$ is the only feasible solution.
\end{proof}

\section{An $L_p$ norm based algorithm to compute the sparsity pattern}
\label{sec:lp_pattern}

In this section, we describe an algorithm for determination of sparsity
pattern in detail, analyze its computational complexity, state and prove
some of its properties, and provide the rationale behind a few subtle
technical issues.

\subsection{Overview of the sparsity pattern algorithm}

We showed in~\Sec{sec:exist_unique} that the minimization problem to
compute a sparse $X$ has a unique solution for any imposed sparsity
pattern.  However, the sparsity determination phase is separate than the
minimization phase and it is important to have a well-defined and universal
method of choosing a sparsity pattern rather than imposing something
arbitrary.

We had presented an $L_1$ norm based algorithm designed for real square
matrices in~\cite{cite:cj_tma_bj_jcp}. Here we generalize for arbitrary
matrices while also using the $L_p$ norm.

As before, we also provide a single parameter $q$, continuously variable in
the interval $[0,1]$, that allows a user to choose anywhere between the two
extremes of very sparse ($q = 0$) and no extra sparsity ($q = 1$). Using a
larger $q$ will provide a better approximation but will be more dense,
whereas a smaller $q$ will provide a worse approximation but lead to fewer
non-zeros in $X$.

In any case, we eliminate only those entries that are small in magnitude
relative to other larger entries.  Intuitively, this makes sense because
eliminating small matrix entries will perturb the matrix and its spectrum
relatively less.

In choosing
entries with large magnitude, our algorithm compares entries in the same
row and column rather than with all other matrix entries.  This choice helps
in obtaining many important properties of the algorithm.  The details and
rationale are provided later in Section~\ref{sec:lp_pat_properties}.

\subsection{Definition of an $L_p$ norm based sparsity for vectors}

Before describing the sparsity pattern algorithm for a matrix, we describe
it for a single vector $x \in \Cm$.  This will be the key ingredient for
the algorithm acting on a matrix.
\begin{definition}
For any $p \in [0, \infty]$, we define the $L_p$ ``norm'' $\norm{x}_p$ for
any vector $x$.
$$
\norm{x}_p = \left\{
     \begin{array}{lcl}
     \displaystyle
       \left(\sum_{i=1}^m \abs{x_i}^p \right)^{\frac{1}{p}} &:& 1 \le p < \infty \\
     \displaystyle
       \max_{i} \abs{x_i} &:& p = \infty\\
     \displaystyle
       \sum_{i=1}^m \abs{x_i}^p &:& 0 <  p < 1\\
     \displaystyle
       \text{number of non-zero } x_i &:& p = 0
     \end{array}
   \right.
$$
\end{definition}
Note that $\norm{x}_p$ is a ``norm'' only when $p \in [1,\infty]$.  It is
not a ``norm'' for $p < 1$.  We simplify the notation and call it a norm
always.

We want to compute a sparsity pattern $\pat{x}$ for $x \in \Cm$. $\pat{x} \in \Rm$.
It contains zeros and ones only. A zero in \patsym means the entry at the
corresponding position is fixed to be zero otherwise it is free to be
modified.

\subsubsection{A combinatorial minimization problem for vector sparsity pattern}

We state below a combinatorial minimization problem for computing $\pat{x}$
so that the number of eliminated entries is maximized while the $L_p$
norm of the eliminated entries is bounded from above.  It is a
combinatorial optimization problem because we can place ones and zeros in
$\pat{x}$ at arbitrary locations.  In addition, the number of non-zeros in
$\pat{x}$ is unknown a priori.

We specify the input parameters -- $p \in [0,\infty]$, $q \in [0,1]$,
and $N \in [0,\norm{x}_0]$.  Here $N$ refers to the minimum number of
non-zeros to be preserved.  It makes sense to have the upper limit for $N$
not greater than $\norm{x}_0$, which is the number of non-zeros present in
$x$. The necessity of $N$ will be discussed in
\Sec{sec:nullity_interaction}.  The symbol `$\circ$' denotes entry-wise
product of two entities.

Given $x, p, q,$ and $N$, compute $\pat{x}$ using
\begin{equation}
\label{eq:Zx_opt}
\begin{array}{l}
\displaystyle \max_{\pat{x}} \norm{x - x \circ \pat{x}}_0 \text{ such that }\\
\displaystyle \qquad (\pat{x})_i = 0 \text{ or } 1, \\[5pt]
\displaystyle \qquad \norm{\pat{x}}_0 \ge N, \text{ and } \\[5pt]
\displaystyle \qquad \norm{x - x \circ \pat{x}}_p \le (1 - q) \norm{x}_p
\end{array}
\end{equation}
If $x = 0$, we define $\pat{x} = 0 \in \Rm$.

Before presenting an algorithm for computing $\pat{x}$, we remark on a few subtle cases.

\begin{remark}
We allow $q$ to take the extreme values, zero or one, in the definition,
but in practical cases it will take an intermediate value usually in $[0.5,1)$.
\end{remark}

The form of the $q$ related inequality condition in our problem is
motivated by two reasons $-$ maintaining a role for $q$ in $p = \infty$
case and better discrimination power for large $p$ case.  We explain this
in the remarks below.

\begin{remark}
One reason we define the $q$ related inequality in terms of $L_p$ norm of
eliminated entries $(\norm{x - x \circ \pat{x}}_p)$ and not in terms of the
norm of preserved entries $(\norm{x \circ \pat{x}}_p)$ has to do with the
large $p$ case, especially $p = \infty$.  Assume $p = \infty$ and $q < 1$
and we impose $\norm{x \circ \pat{x}}_p \ge q \norm{x}_p$ instead.  In this
case, preserving the entry (or entries) with the maximum magnitude would
mean that the norm of preserved entries $\norm{x \circ \pat{x}}_p$ is
greater than $q \norm{x}_p$.  Since this is true for any $q < 1$, that
parameter becomes useless.  This is avoided in~\Eq{eq:Zx_opt} by choosing a
different inequality.
\end{remark}

\begin{remark}
Another reason for using the norm of eliminated entries is when $p$ is
large but not necessarily infinite.  In such a case, moving the smallest
entry in the preserved part to eliminated part will affect the norm of
eliminated part much more.  Thus, comparing the norm of eliminated part
with total norm has more ``discrimination power'', specially for large $p$.
For small $p$, close to one or even less than it, norms of both parts $-$
eliminated and preserved $-$ are affected roughly equally by such a move.
\end{remark}

\subsection{An algorithm to compute sparsity for vectors}
\label{sec:sp_algo_vectors}

We now present an algorithm to compute the sparsity pattern $\pat{x}$ for
the problem posed in~\Eq{eq:Zx_opt}.  See \Alg{alg:Zvec}.
\begin{algorithm}

\label{alg:Zvec}

 \KwData{$x \in \Cm, p \in [0,\infty], q \in [0,1],$ and $N \in [0,\norm{x}_0]$}
 
 \KwResult{Sparsity pattern $\pat{x} \in \Rm$ with $(\pat{x})_i = $ 0 or 1}\vspace{0.5cm}

 $(\pat{x})_i \gets 0$ for $i = 1,\ldots,m$\;
 
 $j \gets num\_zeros(x)$\;
 \If{$j = m$}{
   return\;
 }{}
 
 $x$ $\gets \mbox{abs}(x)$\;
 $ids \gets [1,2,\ldots,m]$\;
 
 Sort $(x,ids)$ pairs in an ascending order with $x_i$ values as the keys\;
 
 \vspace{0.4cm}
 \CommentSty{// Eliminate entries until a threshold is not crossed\\}
 \vspace{0.4cm}
 \While{$j < m - N$}{
 \eIf{$\norm{x(1:j+1)}_p > (1-q)\norm{x}_p$}{
   break\;
   }{
   $id \gets ids(j)$\;
   $(\pat{x})_{id} \gets 1$\;
   $j \gets j + 1$\;
   }
 }
 \vspace{0.4cm}
 \CommentSty{// Preserve the remaining entries\\}
 \vspace{0.4cm}
 \While{$j \le m$}{
   $id \gets ids(j)$\;
   $(\pat{x})_{id} \gets 1$\;
   $j \gets j + 1$\;
 }
\caption{An algorithm to compute sparsity pattern of a given real or
complex vector.  See~\Eq{eq:Zx_opt} for details.}
\end{algorithm}

\begin{remark}
If there are multiple candidate entries in $x$ with the same magnitude that
can lead to a non-unique $\pat{x}$, then any one of them can be used. This
is a ``corner case'' and a simple rule like using all the equal values,
even if fewer than all are sufficient, will break the tie and give a
deterministic algorithm.  In practical problems, we don't expect that the
vector will have a huge number of equal and large entries of exactly equal
magnitude.
\end{remark}

\begin{theorem}
\Alg{alg:Zvec} solves the optimization problem posed in~\Eq{eq:Zx_opt} to
find the sparsity pattern for a given vector.
\end{theorem}
\begin{proof}
We start with the three constraints in~\Eq{eq:Zx_opt}.  Obviously, the
output $\pat{x}$ is such that $(\pat{x})_i = 0 \text{ or } 1$.  The second
``while'' loop ensures that $\norm{\pat{x}}_0 \ge N$.  The first while loop
makes sure that $\norm{x - x \circ \pat{x}}_p \le (1 - q) \norm{x}_p$ and
the second one can only decrease $\norm{x - x \circ \pat{x}}_p$.

Hence all the conditions are maintained and the issue is whether the output
pattern maximizes the number of eliminated entries.  It is clear from the
algorithm that if a particular entry is eliminated, all entries smaller
than that must have been eliminated.  Assume that we want to eliminate one
more (non-zero) entry.  Doing this will violate the either
the second or the third constraint in~\Eq{eq:Zx_opt} depending on which of
the two was the active one in limiting the number of ones in $\pat{x}$.
\end{proof}

\subsubsection{Computational complexity for sparsity pattern of vectors}
\label{sec:Zvec_complexity}

For a vector of size $m$, \Alg{alg:Zvec} runs in $O(m
\log(m))$ operations.  This assumes that an $O(m \log(m))$ algorithm is used for
sorting.  Hence the overall complexity is $O(m \log(m))$.

\subsection{Definition of an $L_p$ norm based sparsity for matrices}
\label{sec:Zmat}

We use the sparsity pattern computation algorithm for vectors described in
\Sec{sec:sp_algo_vectors} to compute a sparsity pattern for
matrices.  The algorithm for matrices can be divided into three separate
steps.  The first two steps can be executed independently and thus in any
order.  Assume we have a matrix $A \in \Cmn$, and parameters $p$, $q$.  We
also need the parameters $N_{row}$ and $N_{col}$, which are used to specify
minimum number of non-zeros in each row and each column, respectively.  In
\Sec{sec:nullity_interaction} we will see how $N_{row}$ and $N_{col}$ are
determined and specified a natural way.  Here are the three steps for
computing $\pat{A}$.
\begin{enumerate}
  \item Compute the pattern in each of the $m$ separate rows (treated as vectors).
  \item Compute the pattern in each of the $n$ separate columns (treated as vectors).
  \item The final pattern $\pat{A}$ is the union, a boolean OR, of the row-based pattern and the
  transpose of column-based pattern.
\end{enumerate}
The algorithm presented above is a specific choice that satisfies many
useful theoretical properties.  See~\Sec{sec:lp_pat_properties}.

\begin{remark}
If it is known that the entry-wise absolute value matrix is symmetric, then
we compute either the row or column pattern and take the union with its
transpose.  This shortcut is not necessary but is faster.
Moreover, whether one takes the shortcut or not, the result is that the
pattern is symmetric for matrices that are Hermitian, skew-Hermitian,
or complex-symmetric.
\end{remark}
\begin{remark}
In general, the number of non-zeros will vary in each row and column
depending on the distribution of the magnitudes of the entries.  In this
sense, this is an adaptive method and one does not know the number of
non-zeros a priori for a given set of parameters.
\end{remark}
\begin{remark}
The algorithm works naturally with rectangular and complex matrices.
\end{remark}
\begin{remark}
\label{rmk:real_pattern}
We compute a single pattern matrix for complex matrices rather than
separate independent patterns matrices for real and imaginary parts for two
distinct reasons $-$ one theoretical and one practical. Firstly, when a
single pattern is computed, the pattern remains invariant if the input
matrix is multiplied by a non-zero complex number
(see \Prop{prop:Zmat_homogeneous} in~\Sec{sec:lp_pat_properties}).  This
property will not hold in general when two patterns are computed. Secondly,
storing two patterns for the sparsified $X$, for the real and imaginary
parts, increases the storage costs and one cannot use complex arithmetic
for first-order optimality conditions (see~\Sec{sec:first_order_opt}).
\end{remark}

\subsubsection{Computational complexity for sparsity pattern of matrices}

We discuss the computational complexity for each of the three steps in
\Sec{sec:Zmat} using the computational complexity for sparsity pattern of
vectors in~\Sec{sec:Zvec_complexity}.

\begin{enumerate}
  \item Computing the sparsity pattern for $m$ rows requires
  $m ( n \log(n) )$ operations in total.
  \item Computing the sparsity pattern for $n$ columns requires
  $n ( m \log(m) )$ operations in total.
  \item Computing the union operation is relatively more complex.  Creating
  the transpose of a sparse matrix takes $O(m n)$ operations.  After transposing,
  the union is computed row by row.  Computing
  union of the patterns of two row vectors of size $n$ each
  takes $O(n \log(n))$ operations.  This is done $m$ times.
\end{enumerate}
Hence, the overall cost is still $O(m n (\log(m) + \log(n)))$ which can be
expressed as $O(m n \log(m n))$.

\subsubsection{Interaction of sparsity and null-space}
\label{sec:nullity_interaction}

We mentioned earlier that we can specify the parameters $N_{row}$ and
$N_{col}$ when computing the sparsity pattern of a matrix. The two values
are used to specify minimum number of non-zeros in each row and each
column, respectively.  Here we show why one needs these parameters.

Consider a matrix $X \in \Cmn$ with a given sparsity pattern $\pat{A}$.
When solving the optimization problem posed in~\Sec{sec:misfit}, the
unknown entries in each row of $X$ have to satisfy $p_R$ linear homogeneous
equality constraints (see~\Sec{sec:notation}).  If the number of allowed
non-zero entries in a row of $\pat{A}$ is less than or equal to $p_R =
\text{dim}(\nullsp{A})$, then all the entries must be zero, and thus the
whole row is zero.  This is an undesirable situation.  The same logic
applies to columns and left null-space.  Thus, an algorithm that decides
the sparsity pattern should also keep sufficient number of non-zeros in
each row and each column so that such degenerate matrix is not produced.
In practice, we choose $N_{row} = \min(n, p_R + 1)$ and $N_{col} = \min(m,
p_L + 1)$.

This restriction implies that the null-space dimension must be known before
computing the sparsity pattern.  Another implication is that sparsification
while maintaining null-space constraints cannot be very useful if a matrix
is highly rank-deficient.  In many applications the rank-deficiency is a
small constant independent of matrix size and this is not a huge concern.

\subsection{Properties of the $L_p$ norm based matrix sparsity patterns}
\label{sec:lp_pat_properties}

We enumerate a few important properties of the sparsity patterns
generated by the $L_p$ norm based algorithm of~\Sec{sec:Zmat}.  Let $A \in
\Cmn$, and $\pat{A} \in \Rmn$ denote its sparsity pattern matrix.  The
number of non-zero entries in $\pat{A}$ is expressed as $\abs{\pat{A}}$.
When we want to discuss a specific parameters $p$ and $q$, we write
$\pat{A;p,q}$ instead.

We also need the notion of complex permutation matrices~\cite[Section IV.1]{cite:bhatia}.
\begin{definition}
\label{def:complex_perm}
A complex permutation matrix is a matrix such that it has a only one non-zero
entry in each row and each column and every non-zero entry is a complex
number of modulus 1.
\end{definition}
Complex permutation matrices are always square and unitary.  We use the
letters $P$ and $Q$ to denote them.

It can be shown that $\pat{A}$ satisfies the following properties.  The
proofs are elementary and we skip them.
\begin{enumerate}[{P}-1]
  \item $A_{ij} = 0 \implies (\pat{A})_{ij} = 0$.
  \item $\pat{\alpha A} = \pat{A}$ for $\alpha \in \complex \setminus \{0\}$. \label{prop:Zmat_homogeneous}
  \item $\pat{A^T} = (\pat{A})^T$.
  \item $\pat{A^*} = (\pat{A})^T$.
  \item $\pat{A}$ does not depend on the signs of entries of $A$.
  \item $\pat{P A Q} = \abs{P} \pat{A} \abs{Q}$, where $P, Q$ are any size-compatible complex permutation matrices
  and \abs{\cdot} denotes entry-wise modulus.
  \item $q_1 < q_2 \implies \pat{A;p,q_1} \leq \pat{A;p,q_2}$ entry-wise, where
  $q_1$ and $q_2$ are any two sparsity parameters.
\end{enumerate}

\section{A priori and a posteriori bounds related to the misfit}
\label{sec:bounds}

Our goal in this section is to prove three theoretical bounds 
related to the sparsification algorithm.

\begin{definition}
\label{def:Apq}
For a given $A \in \Cmn$ and the $L_p$ norm sparsity threshold parameter $q$,
$$
A^{pq} := \pat{A;p,q} \circ A.
$$
\end{definition}
Here `$\circ$' denotes entry-wise multiplication.  Thus, $A^{pq} \in \Cmn$
is the matrix which is obtained by setting those entries of $A$ to zero
that correspond to zero values in the pattern $\pat{A;p,q}$.

For proving the bounds below, we restrict $p$ to be in $[1,\infty]$ so that the
standard $L_p$ norm related inequalities are applicable. For $A \in \Cmn$, we will show
that
$$
\norm{A - A^{pq}}_2 \le (m n)^{\frac{C}{2}} (1-q) \norm{A}_2
$$
where
\begin{equation}
\label{eq:Cpower}
C := 1-\frac{1}{p} + \abs{\half - \frac{1}{p}}.
\end{equation}
A few numerical experiments show that this upper bound is not too loose
when $p$ is close to 1.  However, it is quite pessimistic
as $p$ gets larger than 2.  Our purpose is here is solely to show that by sparsifying
each row and each column individually we can bound the perturbation error
for the full matrix.

For a square non-singular matrix $A$ and for $X$ that satisfies the
sparsity constraints, we will show that
$$
\min_{X} J(X;A) \leq m^{1 + 2C} (1-q)^2 \kappa(A)^2.
$$

For an arbitrary $A$ and the corresponding misfit-minimizing $X$, we show
that all the non-zero eigenvalues of $X \pinv{A}$ and $\pinv{A} X$ are
within a circle of radius $\sqrt{2 J_{min}}$ centered at $(1,0)$ in the
complex plane.  Here $J_{min}$ is the minimum misfit value.  Provided
$J_{min} < \half$ and that the left and right nullities of $X$ are not
greater than those of $A$, we show that
$$
\max(\kappa(X \pinv{A}), \kappa(\pinv{A} X)) \leq \frac{1 + \sqrt{2 J_{min}}}{1 - \sqrt{2 J_{min}}}.
$$

\subsection{An upper bound for the perturbation}

\begin{theorem}
\label{thm:bound1}
For $A \in \Cmn$, $p \in [1,\infty]$, $q \in [0,1]$, $A^{pq}$ defined in
\Def{def:Apq}, and $C$ expressed in~\Eq{eq:Cpower},
$$
\norm{A - A^{pq}}_2 \le (m n)^{\frac{C}{2}} (1-q) \norm{A}_2.
$$
\end{theorem}

\begin{proof}
We first define the usual dual norm parameter $p'$.
\begin{definition}
The number $p'$ is the {\em dual} of $p$ and is defined by
$$
\frac{1}{p} + \frac{1}{p'} = 1 \implies p' = \frac{p}{p-1}.
$$
\end{definition}
In case of $p = 1$ and $p = \infty$, the appropriate limiting values are
used and $p' = \infty$ and $p' = 1$, respectively.


The way $\pat{A;p,q}$ is computed, by working with each row and each column
(see~\Sec{sec:Zmat}), it is obvious that the following inequalities hold.  We
use the MATLAB notation.
\begin{align*}
\norm{A(i,:) - A^{pq}(i,:)}_p &\le (1 - q) \norm{A(i,:)}_p\\
\norm{A(:,j) - A^{pq}(:,j)}_p &\le (1 - q) \norm{A(:,j)}_p
\end{align*}
Taking maximums on each side, we get the following two inequalities.
\begin{align}
\label{eq:Apq_diff1}
\max_i \norm{A(i,:) - A^{pq}(i,:)}_p &\le (1 - q) \max_i \norm{A(i,:)}_p\\
\label{eq:Apq_diff2}
\max_j \norm{A(:,j) - A^{pq}(:,j)}_p &\le (1 - q) \max_j \norm{A(:,j)}_p
\end{align}
We use the following standard results valid for any matrix $Y \in \Cmn$~\cite[Section 6.3]{cite:highamASNA2}.
\begin{align*}
n^{\frac{1}{p} - 1} \norm{Y}_p    &\le \max_j \norm{Y(:,j)}_p \le \norm{Y}_p\\
m^{\frac{1}{p} - 1} \norm{Y}_{p'} &\le \max_i \norm{Y(i,:)}_p \le \norm{Y}_{p'}.
\end{align*}
Using the substitutions $Y \gets A - A^{pq}$ and $Y \gets A$ separately and
combining the results with~\Eqs{eq:Apq_diff1} and (\ref{eq:Apq_diff2}) we get,
\begin{align*}
n^{\frac{1}{p} - 1} \norm{A - A^{pq}}_p    &\le (1 - q) \norm{A}_p\\
m^{\frac{1}{p} - 1} \norm{A - A^{pq}}_{p'} &\le (1 - q) \norm{A}_{p'}.
\end{align*}

We now make use of the following standard result for $p_1, p_2 \in
[1,\infty]$ and $\theta \in [0,1]$~\cite[Section 6.3]{cite:highamASNA2}.
Let $p_{12}$ be such that
$$
\frac{1}{p_{12}} = \frac{1-\theta}{p_2} + \frac{\theta}{p_1}.
$$
Then,
$$
\norm{Y}_{p_{12}} \le \norm{Y}_{p_1}^{\theta} \norm{Y}_{p_2}^{1 - \theta}.
$$
Choosing $p_1 = p$, $p_2 = p',$ and $\theta = \half$ implies $p_{12} = 2$.
This gives
$$
\norm{A - A^{pq}}^2_{2} \le \norm{A - A^{pq}}_{p} \norm{A - A^{pq}}_{p'} \le
(m n)^{1-\frac{1}{p}} (1-q)^2 \norm{A}_{p} \norm{A}_{p'}.
$$
We now need a bound for $\norm{A}_p \norm{A}_{p'}$ in terms of
$\norm{A}_2$.  Using the results in ~\cite[Section 6.3]{cite:highamASNA2}
again, it can be shown that
$$
\norm{A}_{p} \norm{A}_{p'} \le (m n)^{\abs{\half - \frac{1}{p}}} \norm{A}_2^2.
$$
Using this inequality and taking square roots, we get the result we started to
prove.
$$
\norm{A - A^{pq}}_2 \le (m n)^{\frac{C}{2}} (1-q) \norm{A}_2
$$


\end{proof}

\subsection{An upper bound for the misfit functional}

\begin{theorem}
For a square non-singular matrix $A$ and for $X$ that satisfies the
sparsity constraints,
$$
\min_{X} J(X;A) \leq m^{1 + 2C} (1-q)^2 \kappa(A)^2.
$$
where $C$ is defined in~\Eq{eq:Cpower}.
\end{theorem}
\begin{proof}

For $n = m$, \Thm{thm:bound1} gives
$$
\norm{A - A^{pq}}_2 \le m^C (1-q) \norm{A}_2.
$$
Let $X$ be the sparse matrix minimizing the misfit functional $J(X;A)$.
The only constraints on $X$ are due to sparsity.  The matrix $A^{pq}$
satisfies the sparsity constraint as well.  We want to find an upper bound
on $J(X;A)$ in terms of $p$, $q$, and $A$ given that $\norm{A - A^{pq}}_2$
is bounded.
\begin{eqnarray*}
J(X;A) \leq J(A^{pq};A) &=& \half \norm{(A^{pq} - A) \inv{A}}_F^2 + \half \norm{\inv{A} (A^{pq} - A)}_F^2\\
                     &\leq& \frac{m}{2} \left( \norm{(A^{pq} - A) \inv{A}}_2^2 + \norm{\inv{A} (A^{pq} - A)}_2^2 \right)\\
                     &\leq& \frac{m}{2} \left( \norm{A^{pq} - A}_2^2 \norm{\inv{A}}_2^2 + \norm{\inv{A}}_2^2 \norm{A^{pq} - A}_2^2 \right)\\
                        &=& m \norm{A^{pq} - A}_2^2 \norm{\inv{A}}_2^2\\
                     &\leq& m^{1 + 2C} (1-q)^2\; \norm{A}_2^2 \norm{\inv{A}}_2^2\\
                        &=& m^{1 + 2C} (1-q)^2 \kappa(A)^2
\end{eqnarray*}
Taking minimum with respect to $X$ proves the result.
\end{proof}

This analysis is useful in showing that apart from the sparsity parameter
$q$, the condition number of $A$ will likely play an important role too.
That is to say that if $A$ is not very well conditioned, one would
require a denser $X$, by using a larger $q$, to obtain a smaller misfit
between $X$ and $A$.

\begin{remark}
Note that we first bound the Frobenius norm differences in terms of
spectral norm differences and then use the submultiplicative norm property
instead of doing it the other order.  This is because doing it the other
way would have introduced another power of $m$ and led to a less tighter
bound.
\end{remark}

\begin{remark}
The bound proved above is pessimistic for two reasons.  First, it
does use the fact that $X$ is the minimizer but it does not quantify the
effect of minimization.  Second, it makes the worst case assumptions in
using the matrix norm equivalence relations.
\end{remark}

\subsection{A posteriori bounds related to clustering and conditioning}

We now relate the minimum value of the misfit functional $J$ to generalized
condition numbers and eigenvalues of $X \pinv{A} \in \Cmm$ and $\pinv{A} X
\in \Cnn$, where $X$ is the misfit minimizing matrix.  Note that $X$ also
satisfies the null-space related constraints.  These bounds do not depend
on any specific chosen sparsity pattern.  They are pessimistic bounds and
are useful for qualitative understanding.

Define $J_{min}$ to be minimum value of misfit for any fixed chosen sparsity
pattern.
\begin{theorem}
\label{thm:Jmin_radius}
All the non-zero eigenvalues of $X \pinv{A}$ and $\pinv{A} X$ are within a
circle of radius $\sqrt{2 J_{min}}$ centered at $(1,0)$ in the complex
plane.
\end{theorem}
\begin{proof}
As shown in~\Sec{sec:notation}, the matrix $A$ can be factorized as $U_1
\Sigma_r V_1^*$.  Since $X$ satisfies the null-space related constraints,
it can be expressed as $U_1 Y V_1^*$, where $Y \in \Crr$.  We can simplify
$J_{min}$.
$$
J_{min} = \half ||Y \inv{\Sigma_r} - I||_F^2 + \half ||\inv{\Sigma_r} Y - I||_F^2
$$
Thus, $||Y \inv{\Sigma_r} - I||_F$ and $||\inv{\Sigma_r} Y - I||_F$ are
both less than or equal to $\sqrt{2 J_{min}}$.  We use the fact that the
magnitude of each eigenvalue of a matrix is less than the Frobenius norm.
Thus, all eigenvalues of $Y \inv{\Sigma_r} - I$ and $\inv{\Sigma_r} Y - I$
have a magnitude less than $\sqrt{2 J_{min}}$.  This means all eigenvalues
of $Y \inv{\Sigma_r}$ and $\inv{\Sigma_r} Y$ are within a circle of radius
$\sqrt{2 J_{min}}$ around $(1,0)$.

It is trivial to show that ignoring any zero eigenvalues, the eigenvalues of $X \pinv{A}$ and $\pinv{A} X$
are same as the eigenvalues of $Y \inv{\Sigma_r}$ and $\inv{\Sigma_r} Y$,
respectively.  This completes the proof.
\end{proof}

The second result relates generalized condition numbers of
$X \pinv{A}$ and $\pinv{A} X$ with the value of $J_{min}$.
\begin{theorem}
Let $X$ be the minimizing matrix that satisfies all the constraints in
\Eq{eq:min_J} and $J_{min}$ be the minimum value.  Provided $J_{min} <
\half$ and that the left and right nullities of $X$ are not greater than
those of $A$,
$$
\max(\kappa(X \pinv{A}), \kappa(\pinv{A} X)) \leq \frac{1 + \sqrt{2 J_{min}}}{1 - \sqrt{2 J_{min}}}
$$
\end{theorem}
\begin{proof}
We show that
$$
\kappa(X \pinv{A}) \leq \frac{1 + \sqrt{2 J_{min}}}{1 - \sqrt{2 J_{min}}}.
$$
The proof for $\kappa(\pinv{A} X)$ is similar and combining the two will
bound their maximum and prove the theorem.

Using the notation and steps in the proof of \Thm{thm:Jmin_radius}, we get
$$
\kappa(X \pinv{A}) = \frac{\sigma_{\max}(Y \inv{\Sigma_r})}{\sigma_{\min}(Y \inv{\Sigma_r})}.
$$
Since the left and right nullities of $X$ are not greater than those of
$A$, $Y$ and hence $Y \inv{\Sigma_r}$ are full-rank matrices. This means
the denominator above is non-zero.

We bound the numerator from above.
\begin{eqnarray*}
\sigma_{\max}(Y \inv{\Sigma_r}) &=& \max_{x \neq 0} \frac{\norm{Y \inv{\Sigma_r} x}}{\norm{x}}\\
                                &=& \max_{x \neq 0} \frac{\norm{(I-(I-Y \inv{\Sigma_r}))x}}{\norm{x}}\\
                             &\leq& \max_{x \neq 0} \frac{\norm{x} + \norm{(I-Y \inv{\Sigma_r}) x}}{\norm{x}}\\
                                &=& 1 + \norm{I-Y \inv{\Sigma_r}}_2\\
                             &\leq& 1 + \norm{I-Y \inv{\Sigma_r}}_F\\
                             &\leq& 1 + \sqrt{2 J_{min}}.
\end{eqnarray*}

We bound the denominator from below assuming $\sqrt{2 J_{min}} < 1$.
\begin{eqnarray}
\label{eq:sigma_min_bound}
\nonumber
\sigma_{\min}(Y \inv{\Sigma_r}) &=& \min_{x \neq 0} \frac{\norm{Y \inv{\Sigma_r}x}}{\norm{x}}\\ \nonumber
                                &=& \min_{x \neq 0} \frac{\norm{(I - (I - Y \inv{\Sigma_r}))x}}{\norm{x}}\\ \nonumber
                             &\geq& \min_{x \neq 0} \frac{ \abs{\; \norm{x} - \norm{(I - Y \inv{\Sigma_r}) x} \; } }{\norm{x}}\\ 
                                &=& \min_{x \neq 0} \abs{1 - \frac{\norm{(I - Y \inv{\Sigma_r}) x}}{\norm{x}}}
\end{eqnarray}
The following chain of inequalities hold.
$$
\frac{\norm{(I - Y \inv{\Sigma_r}) x}}{\norm{x}} \leq \norm{I - Y \inv{\Sigma_r}}_2 \leq
\norm{I - Y \inv{\Sigma_r}}_F \leq \sqrt{2 J_{min}} < 1
$$
This means the~\Eq{eq:sigma_min_bound} can be simplified as follows.
\begin{eqnarray*}
\min_{x \neq 0} \abs{1 - \frac{\norm{(I - Y \inv{\Sigma_r}) x}}{\norm{x}}} &=&
1 - \max_{x \neq 0} \frac{\norm{(I - Y \inv{\Sigma_r}) x}}{\norm{x}}\\
&=& 1 - \norm{I - Y \inv{\Sigma_r}}_2\\
&\geq& 1 - \norm{I - Y \inv{\Sigma_r}}_F\\
&\geq& 1 - \sqrt{2 J_{min}}.
\end{eqnarray*}
Hence
$$
\kappa(X \pinv{A}) \leq \frac{1 + \sqrt{2 J_{min}}}{1 - \sqrt{2 J_{min}}}
$$
and we are done.
\end{proof}

\section{Subspace preserving sparsification}
\label{sec:preserve}
We now show that the output sparse matrix automatically belongs to certain
subspaces, without imposing special constraints for them, if the input
matrix belongs to them. In particular, this property holds for Hermitian,
complex-symmetric, Hamiltonian, circulant, centrosymmetric, and
persymmetric matrices and also for each of the skew counterparts.
Obviously, proving this will require that all the components of the
algorithm -- the sparsity constraints, the null-space related constraints,
the form of the misfit functional -- have features that make the
preservation of subspaces possible.  We first prove the results for general
permutation transformations and then apply the results to specific complex
permutation matrices in~\Sec{sec:invariance_special}.

\subsection{Invariance of the general minimization problem}

We prove a few intermediate results. Here is the required notation.  Let $\alpha \in
\complex, \abs{\alpha} = 1$ be a constant, $P,Q \in \Cmm$ or $\Cnn$ be
complex permutation matrices, and $Y \in \Cmn$.  Let $\h{Y} = Y$ or $Y^*$
or $Y^T$, where $op$ is a placeholder and means operation.
Then $\h{(\h{Y})} = Y$, $\norm{\h{Y}}_F = \norm{Y}_F$, $\pinv{(\alpha P
\h{A} Q)} = (1/\alpha) Q^* \pinv{(\h{A})} P^* = \bar{\alpha} Q^* \pinv{(\h{A})} P^*$, and $\pinv{(\h{A})} =
\h{(\pinv{A})}$.  If $op$ denotes transpose or conjugate transpose $\h{(Y_1
Y_2)} = \h{Y_2} \h{Y_1}$.  The proofs of each of these equalities is
trivial.
\begin{definition}
\label{def:fmap}
Let $f(Y) = \alpha P \h{Y} Q$ for some given $\alpha, P,$ and $Q$ where $\abs{\alpha} = 1$ and $P$ and $Q$
are appropriately sized complex permutation matrices.
\end{definition}
\begin{lemma}
\label{lem:J_invariance}
$J(X;A) = J(f(X);f(A))$, where $J(X;A)$ is the misfit functional defined in
\Eq{eq:J_pinv}.
\end{lemma}
\begin{proof}
Define
\begin{eqnarray*}
J_R(X;A) &=& \half \norm{(X - A) \pinv{A}}^2\\
J_L(X;A) &=& \half \norm{\pinv{A} (X - A)}^2
\end{eqnarray*}
so that $J(X; A) = J_R(X;A) + J_L(X;A)$.  The subscripts $L$ and $R$ stand
for left and right, respectively, and refer to the side on which the
pseudoinverse is applied.  Then
\begin{eqnarray*}
J_R(f(X);f(A)) &=& \half \norm{(\alpha P \h{X} Q - \alpha P \h{A} Q) \pinv{(\alpha P \h{A} Q)}}^2\\
               &=& \half \norm{\alpha P (\h{X} - \h{A}) Q \bar{\alpha} Q^* \pinv{(\h{A})} P^*}^2\\
               &=& \half \norm{(\h{X} - \h{A}) \h{(\pinv{A})}}^2
\end{eqnarray*}
If $op$ denotes transpose or conjugate transpose, the equality above shows
that
$$
\half \norm{(\h{X} - \h{A}) \h{(\pinv{A})}}^2 = \half \norm{\pinv{A}(X - A)}^2 = J_L(X;A).
$$
If $op$ denotes no change (that is $\h{Y} \equiv Y$), then
$$
\half \norm{(\h{X} - \h{A}) \h{(\pinv{A})}}^2 = \half \norm{(X - A) \pinv{A}}^2 = J_R(X;A).
$$
A similar result can be proved by expanding $J_L(f(X);f(A))$. For each
substitution for $op$, $J_R(X;A) + J_L(X;A) = J_R(f(X);f(A)) +
J_L(f(X);f(A))$.  Hence proved.
\end{proof}

We now relate the transformation of null-spaces when matrices $X$ and $A$
are transformed by the mapping $f$.
\begin{lemma}
\label{lem:nullsp_invariance}
Let $X \in \Cmn$ such that $X \nullsp{A} = 0$ and $X^* \nullsp{A^*} = 0$ both hold. Then $f(X) \nullsp{f(A)}
= 0$ and $(f(X))^* \nullsp{(f(A))^*} = 0$.
\end{lemma}
\begin{proof}
The following implications are easy to see.
\begin{eqnarray*}
v_R \in \nullsp{A}        &\implies& A v_R = 0 \mbox{ and } X v_R = 0\\
v_L \in \nullsp{A^*}      &\implies& A^* v_L = 0 \mbox{ and } X^* v_L = 0\\
w_R \in \nullsp{f(A)}     &\implies& \alpha P \h{A} Q w_R = 0\\
w_L \in \nullsp{(f(A))^*} &\implies& \bar{\alpha} Q^* (\h{A})^* P^* w_L = 0
\end{eqnarray*}
We consider the three $op$ cases separately.  Consider the first case when
$op$ denotes no change (that is $\h{Y} \equiv Y$). Then it is evident that
$\nullsp{f(A)} = Q^* \nullsp{A}$ and $\nullsp{(f(A))^*} = P \nullsp{A^*}$.
Thus,
$$
f(X) \nullsp{f(A)} = \alpha P X Q Q^* \nullsp{A} = \alpha P (X \nullsp{A}) = 0
$$
and
$$
(f(X))^* \nullsp{(f(A))^*} = \bar{\alpha} Q^* X^* P^* P \nullsp{A^*} = \bar{\alpha} Q^* (X^* \nullsp{A^*}) = 0.
$$
This proves the result for the first case.

In the second case, where $op$ denotes conjugate-transpose, a similar
argument shows that $\nullsp{f(A)} = Q^* \nullsp{A^*}$ and
$\nullsp{(f(A))^*} = P \nullsp{A}$.  Thus,
$$
f(X) \nullsp{f(A)} = \alpha P X^* Q Q^* \nullsp{A^*} = \alpha P (X^* \nullsp{A^*}) = 0
$$
and
$$
(f(X))^* \nullsp{(f(A))^*} = \bar{\alpha} Q^* X P^* P \nullsp{A} = \bar{\alpha} Q^* (X \nullsp{A}) = 0.
$$
This proves the result for the second case.

In the third case, where $op$ denotes transpose, we can follow similar
steps asn show that $\nullsp{f(A)} = Q^* \overline{\nullsp{A^*}}$ and
$\nullsp{(f(A))^*} = P \overline{\nullsp{A}}$.  Thus,
$$
f(X) \nullsp{f(A)} = \alpha P X^T Q Q^* \overline{\nullsp{A^*}} = \alpha P \overline{X^* \nullsp{A^*}} = 0
$$
and
$$
(f(X))^* \nullsp{(f(A))^*} = \bar{\alpha} Q^* \overline{X} P^* P \overline{\nullsp{A}} = \bar{\alpha} Q^* \overline{X \nullsp{A}} = 0.
$$
This proves the result for the third case and we are done.
\end{proof}

We now relate how the mapping $f$ interacts with the operation $\patsym$
that computes the sparsity pattern (\Sec{sec:Zmat}).
\begin{lemma}
\label{lem:permute_invariance}
Let $A \in \Cmn$ and $f(A) = \alpha P A^{op} Q$ as used earlier, then
$\pat{f(A)} = \abs{P} \pat{A} \abs{Q}$.  Here \abs{\cdot}denotes entry-wise
modulus.
\end{lemma}
\begin{proof}
The proof is immediate by using the various properties of $\patsym$
enumerated in~\Sec{sec:lp_pat_properties}.
\end{proof}

We show how the solution of the full sparsification problem changes when
the input matrix changes.
\begin{theorem}
\label{thm:master}
If $X$ solves~\Eq{eq:min_J_2} for a given input $A$, then $f(X)$ solves it
for the input $f(A)$, where $f$ is given in~\Def{def:fmap}.
\end{theorem}
\begin{proof}
The lemmas \ref{lem:J_invariance}, \ref{lem:nullsp_invariance}, and \ref{lem:permute_invariance}
show that each of the three components of the
minimization problem -- the sparsity constraints, the null-space related
constraints, the form of the misfit functional -- are invariant under
transformation by $f$.  This proves the statement.
\end{proof}

\subsection{Invariance for specific matrix subspaces}
\label{sec:invariance_special}

We now define a few specific complex permutation matrices
(see \Def{def:complex_perm}).  This allows us to succinctly characterize
certain special matrix subspaces.
\begin{definition}
\label{def:matJ}
For each $m > 0$, $J_m$ is a matrix in $\Rmm$ with ones on the anti-diagonal
and zeros elsewhere.
\end{definition}
For example,
$$
J_3 =
\begin{bmatrix}
0 & 0 & 1 \\
0 & 1 & 0 \\
1 & 0 & 0
\end{bmatrix}.
$$
\begin{definition}
\label{def:matK}
For each even $m > 0$, $K_m$ is the following skew-symmetric matrix in $\Rmm$.
$$
K_m :=
\begin{bmatrix}
0 & I_{\frac{m}{2}} \\
-I_{\frac{m}{2}} & 0
\end{bmatrix}
$$
\end{definition}
\begin{definition}
\label{def:matC}
For each $m > 0$,
$$
C^+_m :=
\begin{bmatrix}
0_{m-1} & I_{m-1} \\
1 & 0_{m-1}^T
\end{bmatrix}
\mbox{ and }
C^-_m :=
\begin{bmatrix}
0_{m-1} & I_{m-1} \\
-1 & 0_{m-1}^T
\end{bmatrix}.
$$
where $0_{m-1}$ is the zero vector in ${\reals}^{(m-1) \times 1}$.
\end{definition}
The matrices $C^+_m$ and $C^-_m$ are used for cyclic permutations.
\Tab{tab:mat_spaces} shows how certain matrix subspaces can be
characterized using these permutation
matrices.  The conditions satisfied by the sparsity pattern computed by
\Alg{alg:Zvec} can be easily proved using the properties mentioned in
\Sec{sec:lp_pat_properties}.

\begin{table}
\label{tab:mat_spaces}
\begin{center}
\renewcommand{\arraystretch}{1.1}
\begin{tabular}{ | l | l | l | }
\hline
Matrix subspace        & Condition on $A$      & $Z = \pat{A}$ satisfies      \\ \hline
\hline
centrosymmetric        & $ A J - J A     = 0 $ & $ Z J - J Z               = 0 $ \\ \hline
skew-centrosymmetric   & $ A J + J A     = 0 $ & $ Z J - J Z               = 0 $ \\ \hline
circulant              & $ A C^+ - C^+ A = 0 $ & $ Z C^+ - C^+ Z           = 0 $ \\ \hline
skew-circulant         & $ A C^- - C^- A = 0 $ & $ Z C^+ - C^+ Z           = 0 $ \\ \hline
complex-symmetric      & $ A - A^T       = 0 $ & $ Z - Z^T                 = 0 $ \\ \hline
skew-complex-symmetric & $ A + A^T       = 0 $ & $ Z - Z^T                 = 0 $ \\ \hline
Hamiltonian            & $ K A + A^* K   = 0 $ & $ \abs{K} Z - Z^T \abs{K} = 0 $ \\ \hline
skew-Hamiltonian       & $ K A - A^* K   = 0 $ & $ \abs{K} Z - Z^T \abs{K} = 0 $ \\ \hline
Hermitian              & $ A - A^*       = 0 $ & $ Z - Z^T                 = 0 $ \\ \hline
skew-Hermitian         & $ A + A^*       = 0 $ & $ Z - Z^T                 = 0 $ \\ \hline
persymmetric           & $ A J - J A^*   = 0 $ & $ Z J - J Z^T             = 0 $ \\ \hline
skew-persymmetric      & $ A J + J A^*   = 0 $ & $ Z J - J Z^T             = 0 $ \\ \hline
symmetric (real)       & $ A - A^T       = 0 $ & $ Z - Z^T                 = 0 $ \\ \hline
skew-symmetric (real)  & $ A + A^T       = 0 $ & $ Z - Z^T                 = 0 $ \\ \hline
\end{tabular}
\end{center}
\caption{A summary of conditions on a matrix $A$ so that it belongs to a
particular subspace of $\Cmm$.  Also shown are the conditions satisfied by
the sparsity pattern computed by \Alg{alg:Zvec}.
Definitions \ref{def:matJ}, \ref{def:matK}, and \ref{def:matC}
give the expressions for the matrices $J, K, C^+,$ and $C^-$.}
\end{table}

\begin{theorem}
If the input matrix $A$ for the problem in~\Eq{eq:min_J_2} belongs to one of the
following matrix subspaces -- Hermitian, complex-symmetric, Hamiltonian, circulant,
centrosymmetric, persymmetric, or one of the skew counterparts -- then the
output matrix $X$ belongs to the same subspace.
\end{theorem}
\begin{proof}
If the input $A$ satisfies any of the stated properties, then $A = f(A)$
for a specific choice of $f$ (see~\Def{def:fmap}) for one of the categories
shown in \Tab{tab:mat_spaces}.  If $X$ is the solution of
\Eq{eq:min_J_2} corresponding to $A$, then using \Thm{thm:master}, $f(X)$
is the solution corresponding to $f(A)$.  Since $A = f(A)$, and as proved
in \Thm{thm:unique}, \Eq{eq:min_J_2} has a unique solution, it implies $X =
f(X)$.  Thus $X$ preserves the relevant property satisfied by $A$.
\end{proof}

\section{Numerical results}
\label{sec:num_res}

We now present some sparsification results for a fixed real asymmetric matrix
$A \in \reals^{40\times 40},$ where
\begin{equation}
\label{eq:A_sample}
A_{ij} = \cos(3^{\frac{1}{4}} i^{\frac{1}{2}} j)^5,
\end{equation}
and the indices start at 1.
This is a good candidate matrix since its condition number is approximately
621, which is similar to numbers found in spectral finite element matrices, we don't
have to scale its rows and columns, it is deterministic,
and it has some values that are near zero in magnitude.  See~\Fig{fig:A_sample_and_X}(a)
for a 2-D plot of the matrix generated using the {\tt cspy} program~\cite{Suitesparse}.

We first sparsify it using the following parameters: $p = 1$ and $q = 0.8$.
\Fig{fig:A_sample_and_X}(b) shows the sparsity pattern and values of the output matrix $X$.
We get the following
data for this set of parameters -- $\text{cond}(X) = 552$,
$\text{cond}(\pinv{A} X) = 4.73$, $\text{cond}(X \pinv{A}) = 5.37$, and
$X$ contains 597 non-zeros, which is approximately 37\% density.  Thus, nearly
two-thirds of the entries are removed due to sparsification and the condition
number is not far from 1, the minimal value.

\begin{figure}
\begin{center}
	\subfigure[Input dense $A$]
	{
	\includegraphics[scale=0.5]{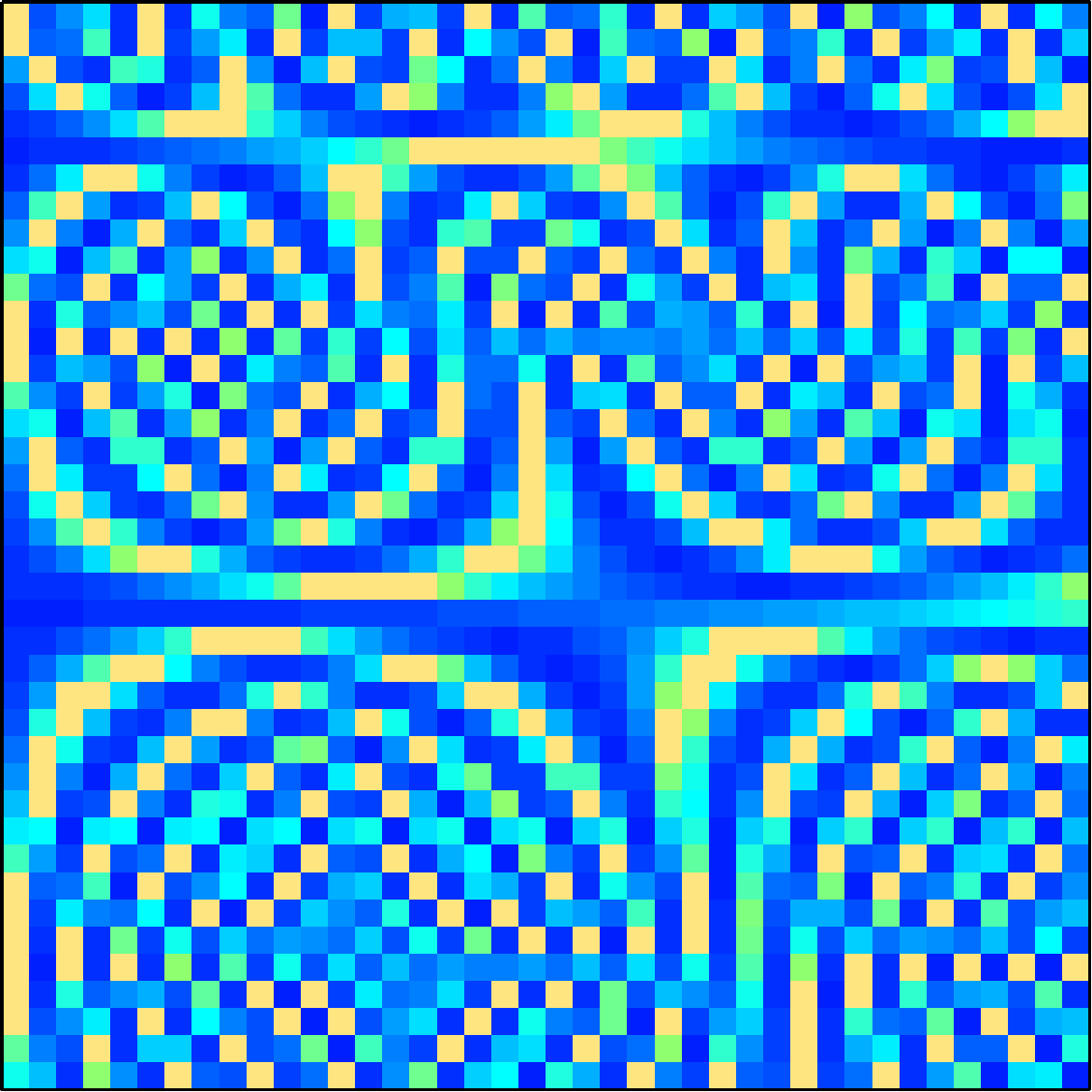}
	}
	\subfigure[Output sparse $X$]
	{
  \includegraphics[scale=0.5]{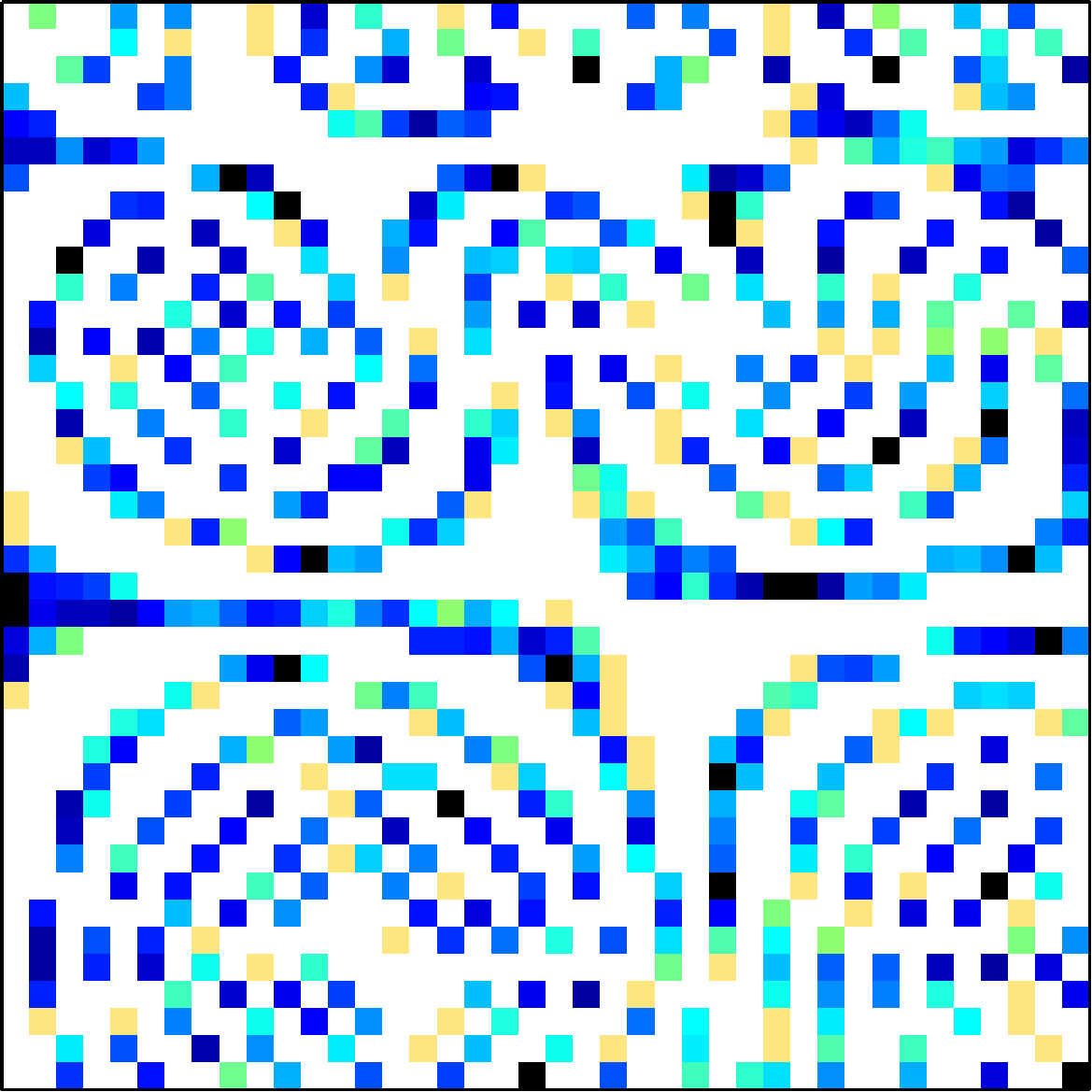}
	}
\caption{Input dense $A \in \reals^{40\times 40}$, specified in~\Eq{eq:A_sample}, and output sparse $X$ for $p = 1$ and $q = 0.8$.  Plots
generated using the {\tt cspy} program~\cite{Suitesparse} where darker pixel values correspond to larger magnitude.}
\label{fig:A_sample_and_X}
\end{center}
\end{figure}

\Fig{fig:svd} shows the singular values of $A$ and $X$.  Clearly,
the lower end of the spectrum is perturbed less than the higher end.
This shows that the algorithm is working as intended.
\begin{figure}
\centering
\includegraphics[scale=0.4]{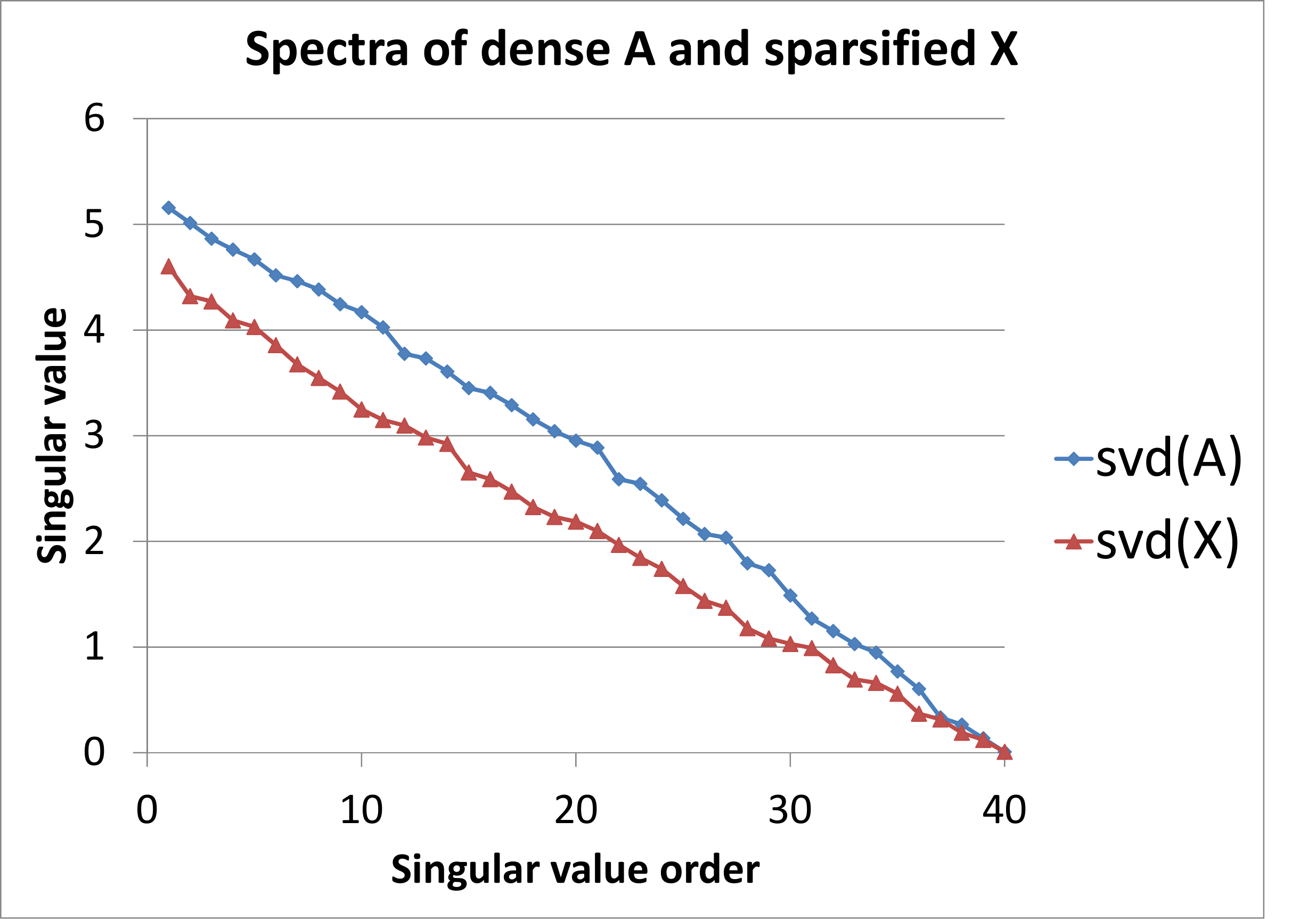}
\caption{Singular values of $A$, specified in~\Eq{eq:A_sample}, and $X$ for $p = 1$ and $q = 0.8$.}
\label{fig:svd}
\end{figure} 

It is desirable to compute three quantities for
measuring sparsification performance.  First is
the condition number of $\pinv{A} X$ or $X \pinv{A}$,
second is the Frobenius norm difference of inverses,
$\norm{\pinv{X} - \pinv{A}}_F / \norm{\pinv{A}}_F$,
and third is the number of non-zeros in $X$.
The exact multiplication order in $\pinv{A} X$ or $X \pinv{A}$
is not particularly important.
This is because almost always the two quantities are close to
each other, as we have observed. \Fig{fig:cond_nnz}(a) shows how the
condition number varies when we vary $q$ and $p$.
Similarly, \Fig{fig:cond_nnz}(b) shows how the
number of non-zeros change on varying the parameters.
Note that for a fixed $q$, increasing $p$ leads
to an increase in the number of non-zeros.
The variation in condition number is reasonably
smooth as long as not too many entries are discarded.
An important measure of how much $X$ deviates from
$A$ is the relative
Frobenius norm difference of inverses.
This is shown in~\Fig{fig:frob_diff}.  The difference
of inverses is important because our goal is to approximate
the action of the inverse and not the operator itself.
\begin{figure}
\begin{center}
	\subfigure[Condition number of $\pinv{A} X$]
	{
  \includegraphics[scale=0.4]{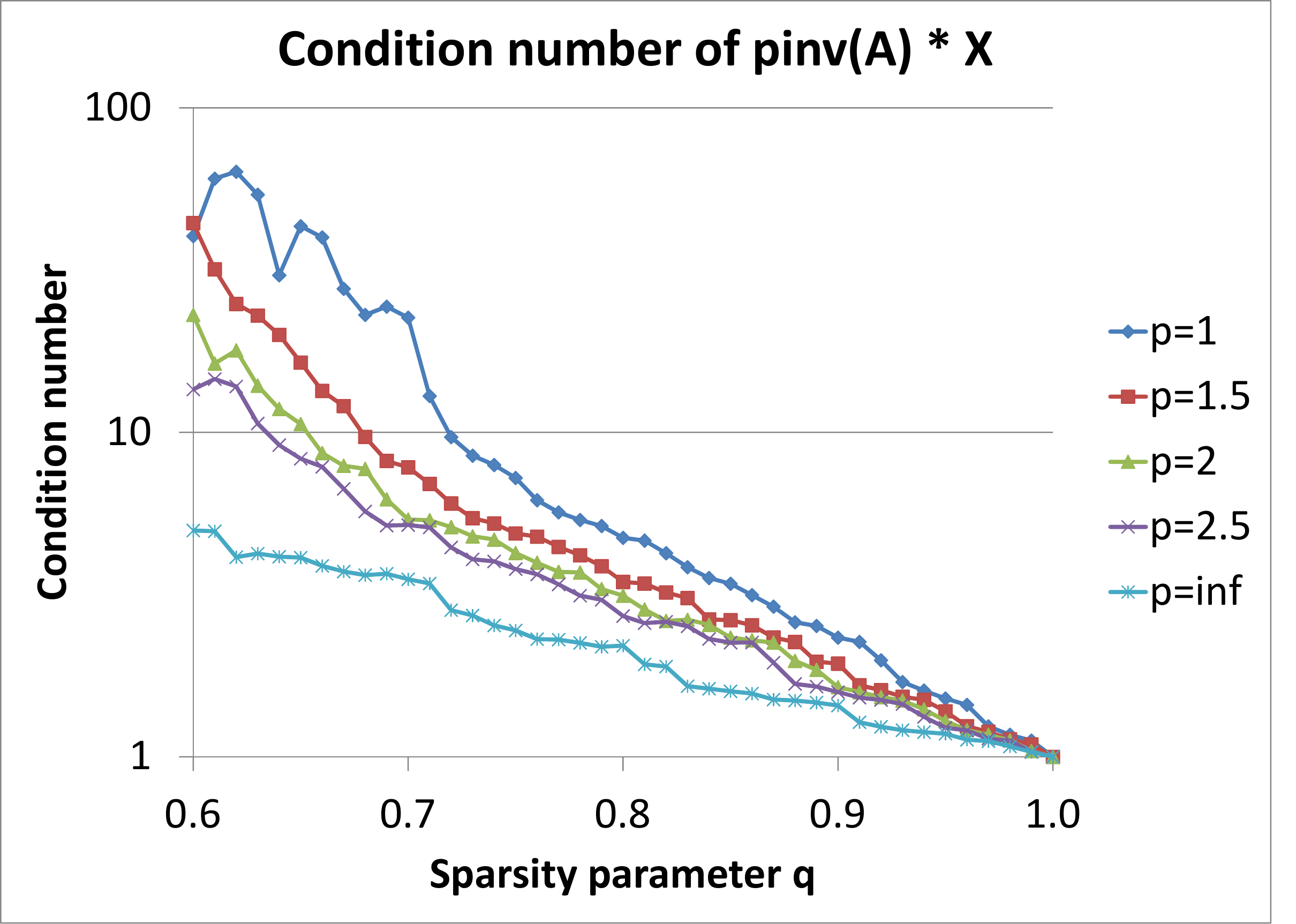}
	}\\[0.5cm]
	\subfigure[Number of non-zeros in $X$]
	{
	\includegraphics[scale=0.4]{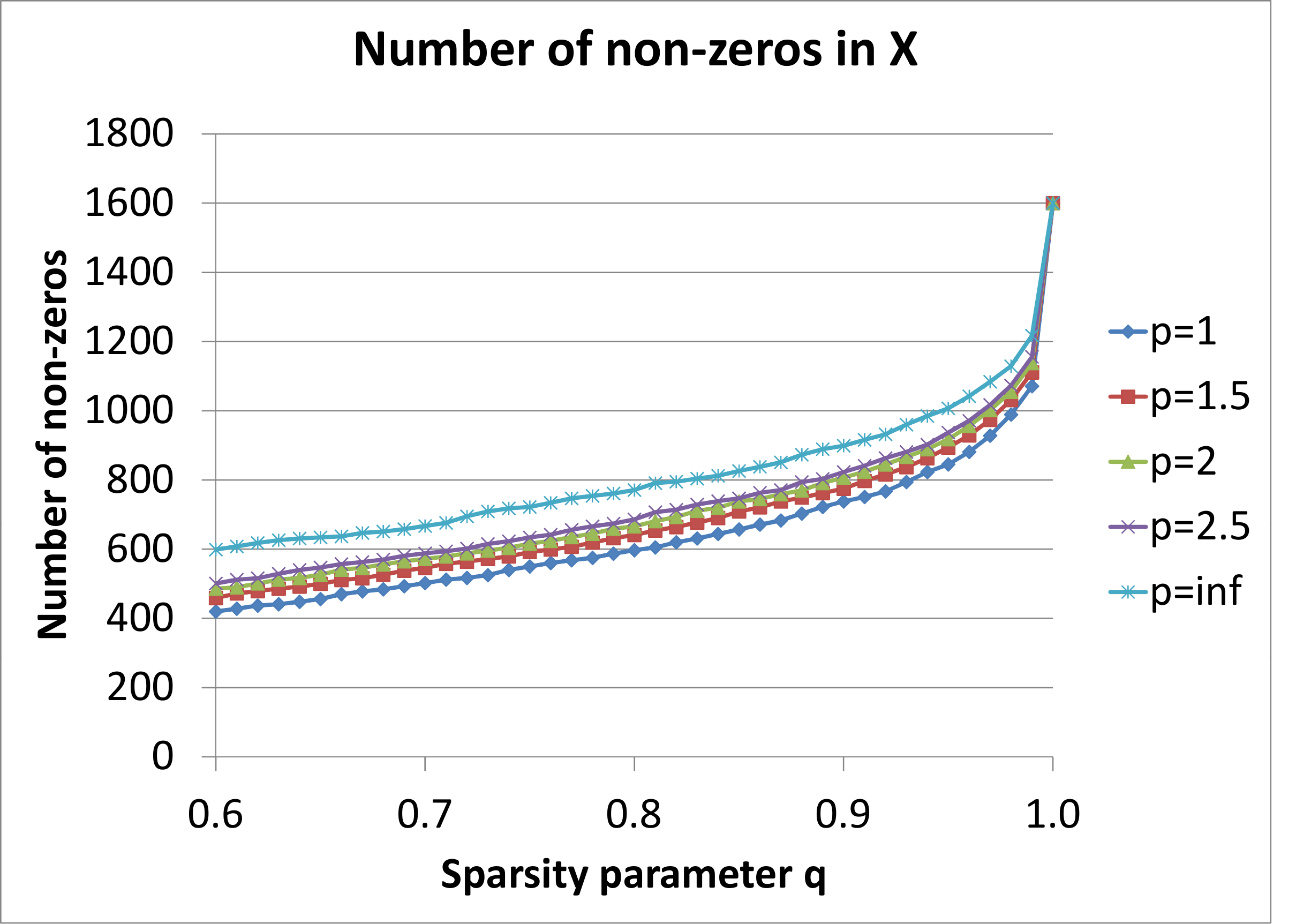}
	}
\caption{Effect of varying $p$ and $q$ on conditioning and sparsity.  The input $A$ is specified in~\Eq{eq:A_sample}.}
\label{fig:cond_nnz}
\end{center}
\end{figure}

\begin{figure}
\centering
\includegraphics[scale=0.4]{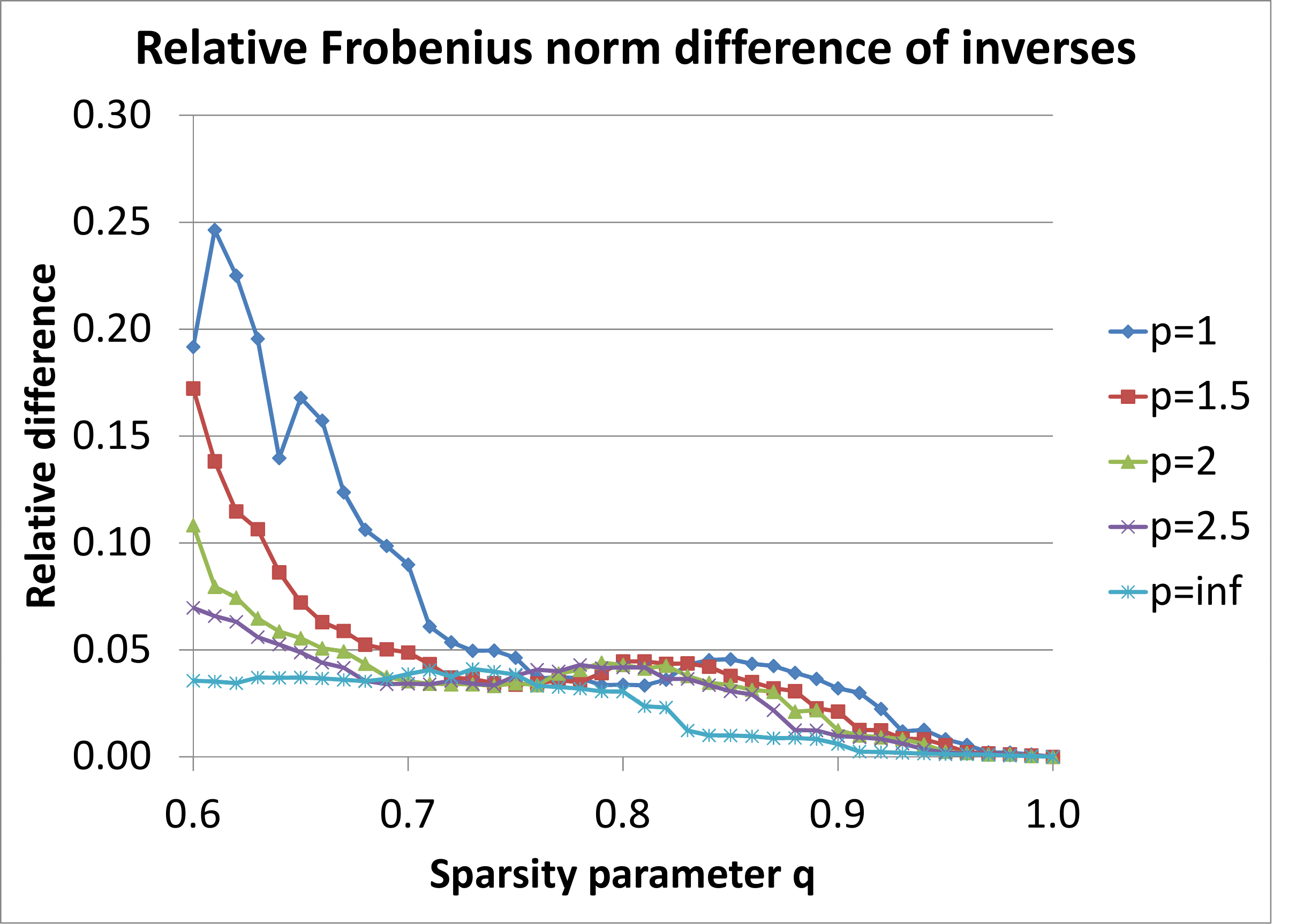}
\caption{Effect of varying $p$ and $q$ on the relative Frobenius norm difference of inverses
$\left(\norm{\pinv{X} - \pinv{A}}_F / \norm{\pinv{A}}_F \right)$.}
\label{fig:frob_diff}
\end{figure} 

Based on the discussion above, it is natural to question
which $p$ to choose.  We argue that the precise value of
$p$ is not important as long as one can change $q$
to achieve a given number of non-zeros.  An evidence is shown
in~\Fig{fig:psensitive} where we plot the data for
various $p$ and $q$ values together.  It shows that the conditioning
and relative difference are highly correlated with the number of non-zeros rather than
the exact $p$ and $q$ values.  The ``curves'' for five
$p$ values lie almost on top of each other when $q$ is varied
in $[0.6,1]$.
\begin{figure}
\begin{center}
	\subfigure[$\text{cond}(\pinv{A} X)$]
	{
     \includegraphics[scale=0.4]{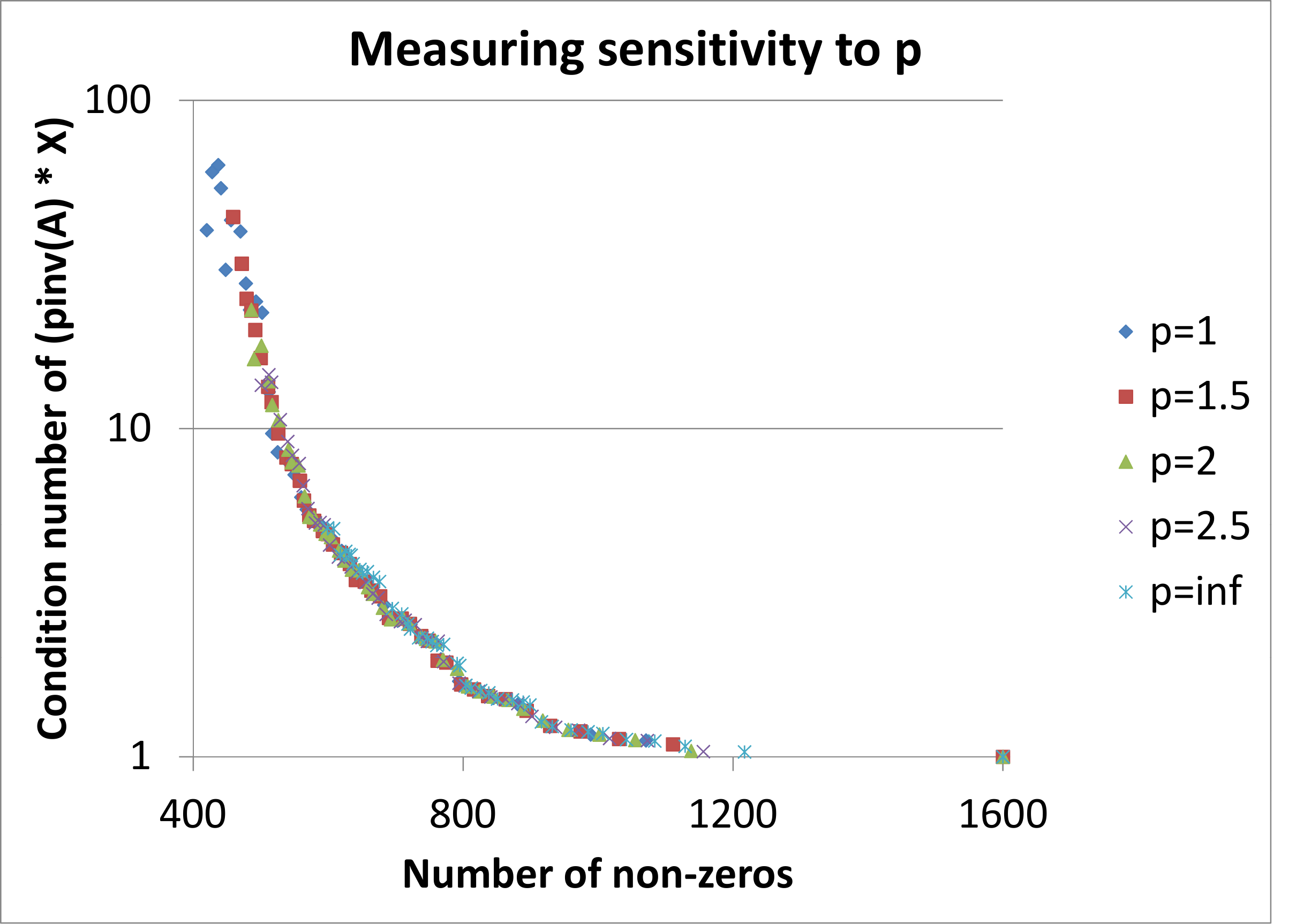}
	}\\[0.5cm]
	\subfigure[$\norm{\pinv{X} - \pinv{A}}_F / \norm{\pinv{A}}_F $]
	{
	\includegraphics[scale=0.4]{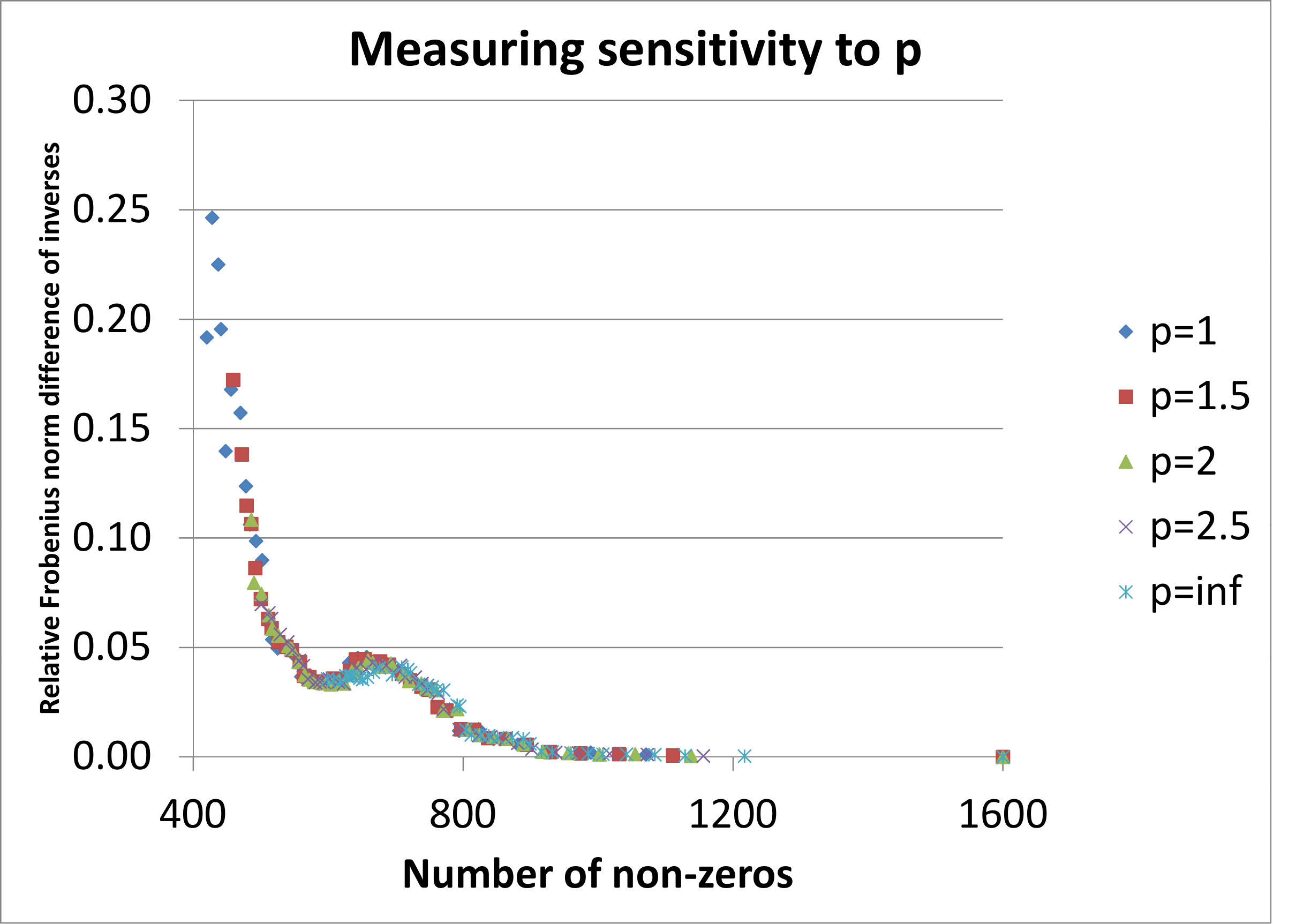}
	}
\caption{The figures show that the choice of $p$ is not too important
if $q$ can be varied to achieve a specific amount of sparsity.  Conditioning
and relative difference of inverses is highly correlated with sparsity and not
with parameters $p$ and $q$ individually.}
\label{fig:psensitive}
\end{center}
\end{figure} 

We show that the computed $X$ is such that the eigenvalues
of $\pinv{A} X$, which are same as the eigenvalues of
$X\pinv{A}$, are clustered around a value near 1 on the
real axis.  See~\Fig{fig:eig}(a) and (b) for eigenvalues of
$A$, $X$, and $\pinv{A} X$.

\begin{figure}
\begin{center}
	\subfigure[Eigenvalues of $A$ and $X$]
	{
  \includegraphics[scale=0.4]{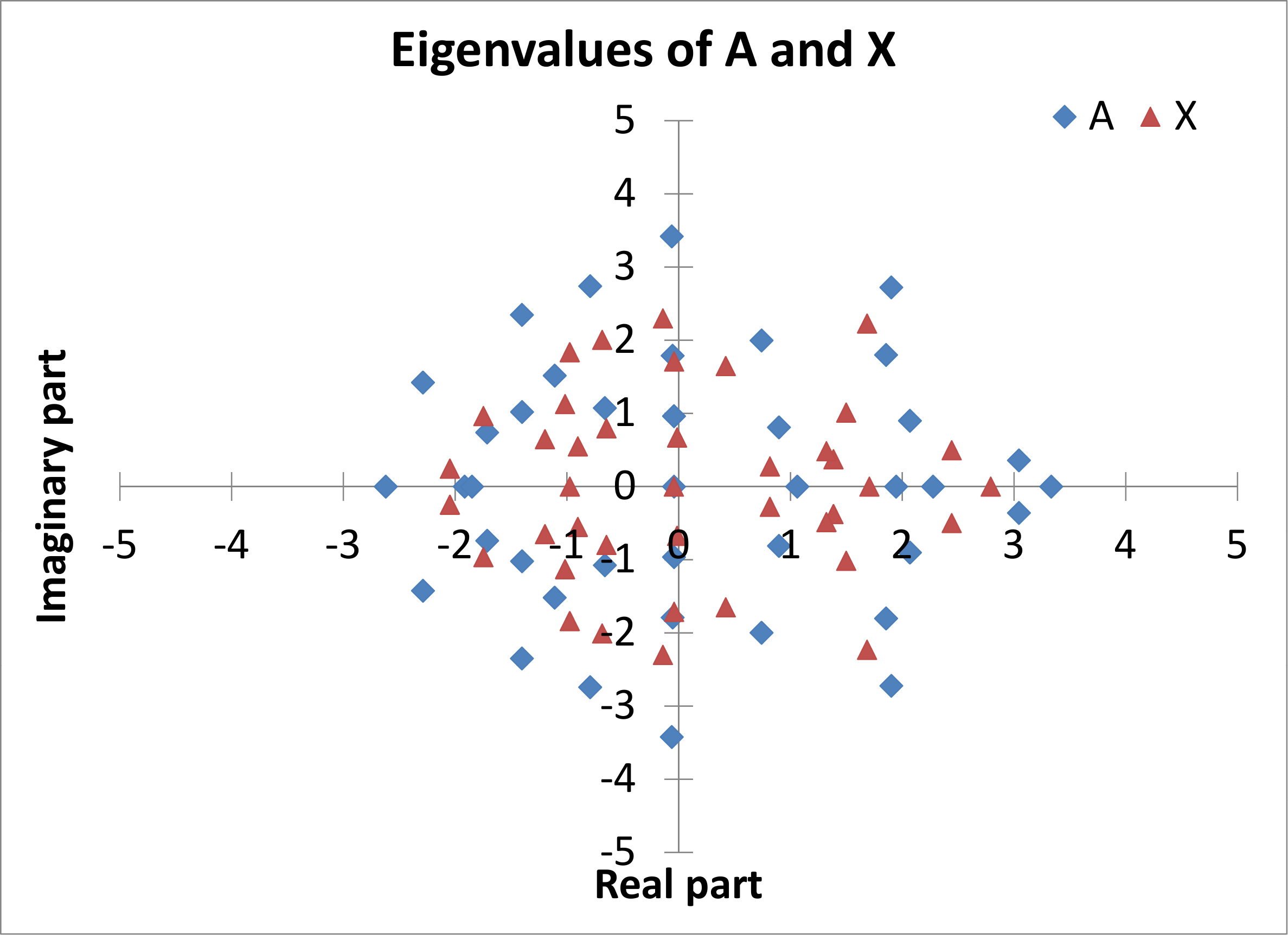}
	}\\[1cm]
	\subfigure[Eigenvalues of $\pinv{A} X$]
	{
	\includegraphics[scale=0.4]{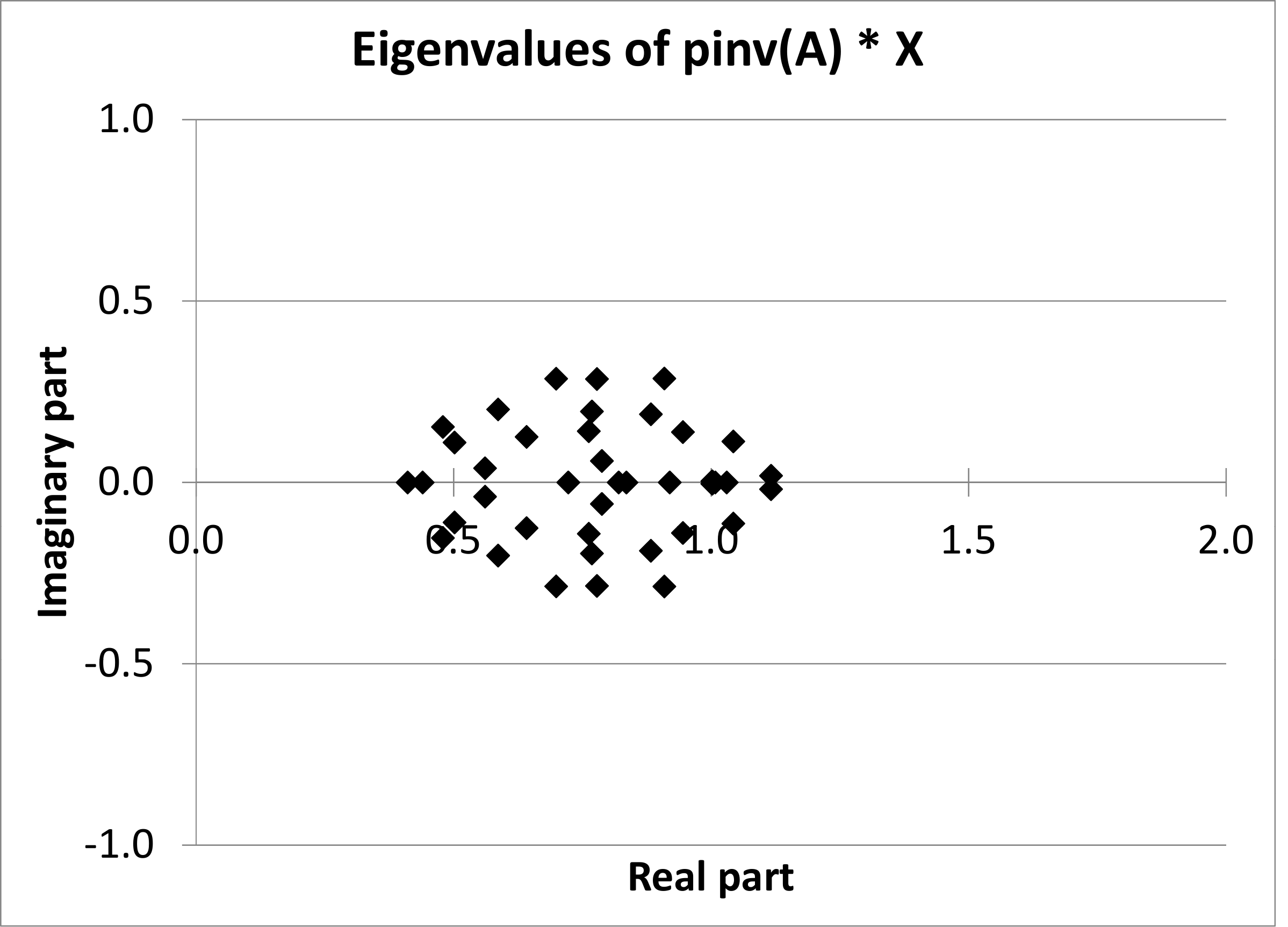}
	}
\caption{Eigenvalues before and after sparsification. The clustered values in the second figure
show that $X$ would perform well for preconditioning $A$ specified in~\Eq{eq:A_sample}.}
\label{fig:eig}
\end{center}
\end{figure}

The last observation we have, which will be useful in
the next paper, is that the non-zero entries of $X$
are highly correlated with the entries of $A$ at the
preserved locations.  \Fig{fig:correlate} shows
the sorted entries of vectorized $A$ and the corresponding
entries in $X$ for $p = 1$ and $q = 0.9$.
\begin{figure}
\centering
\includegraphics[scale=0.4]{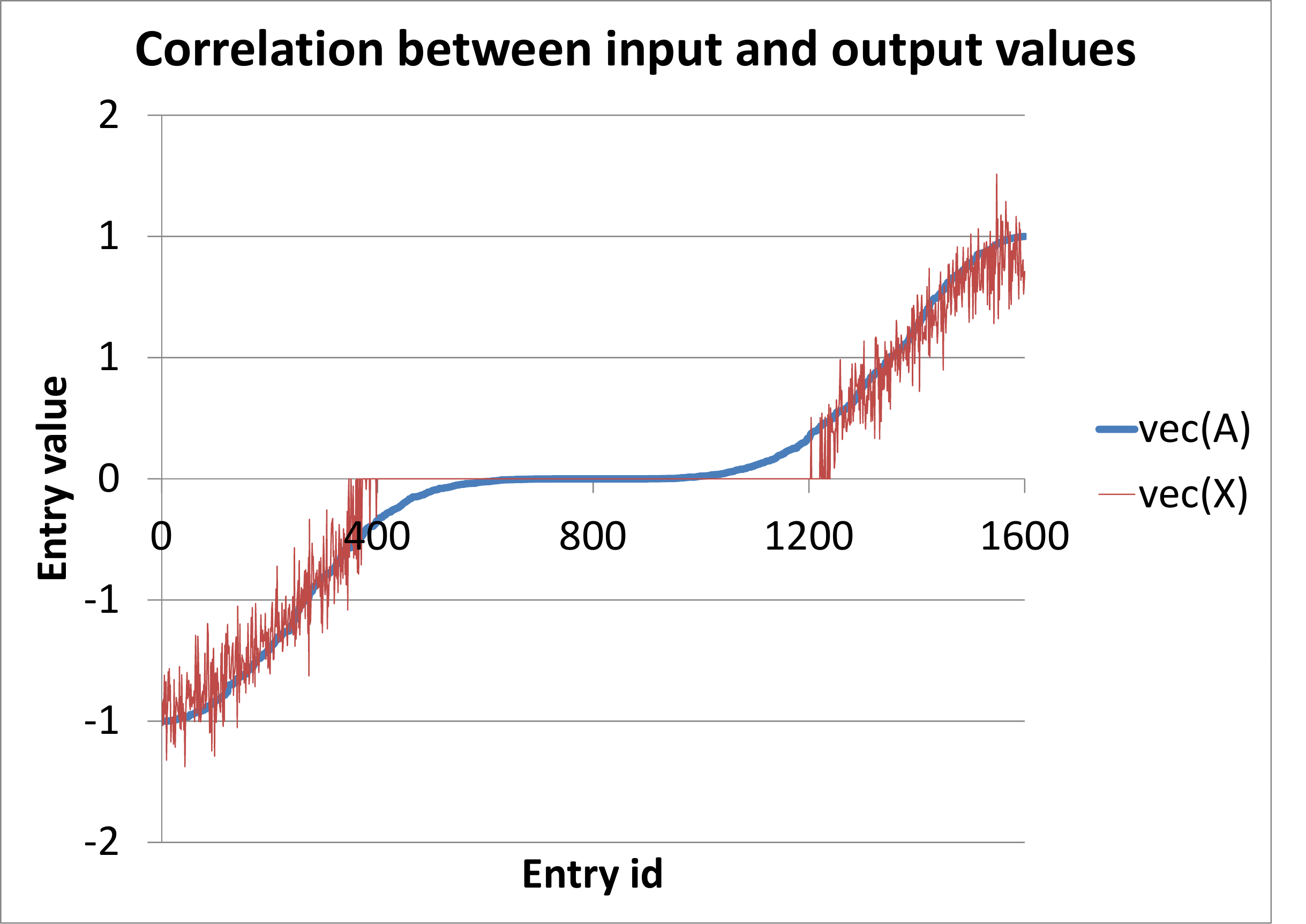}
\caption{Correlation between $A$ and $X$ values after sorting vectorized $A$ and pairing corresponding
$X$ entries with it.}
\label{fig:correlate}
\end{figure} 

\vspace{1cm}

\noindent \textbf{ \large{Acknowledgements}}

This work was partially supported by the US Department of Energy SBIR Grant
DE-FG02-08ER85154.  The author thanks Travis M. Austin, Marian Brezina,
Leszek Demkowicz, Ben Jamroz, Thomas A. Manteuffel, and John Ruge for many
discussions.

\bibliographystyle{elsarticle-num}
\bibliography{01_sparsification}

\end{document}